\documentclass[11pt,twoside, a4paper, english, reqno]{amsart}
 \usepackage[foot]{amsaddr}
\usepackage[dvips]{epsfig}
\usepackage{amscd}
\usepackage{amssymb}
\usepackage{amsthm}
\usepackage{amsmath}
\usepackage{latexsym}

\usepackage{upref}
\usepackage{bm}
\usepackage{color}
\usepackage{hyperref}
\usepackage{graphicx}
\usepackage{lipsum}
\usepackage{comment}
\usepackage{romannum}

\hypersetup{linkcolor=blue, colorlinks=true
,citecolor = red}

\theoremstyle{plain}
\newtheorem{Th}{Theorem}[section]
\newtheorem{Lem}[Th]{Lemma}
\newtheorem{Cor}[Th]{Corollary}

\theoremstyle{definition}
\newtheorem{Def}[Th]{Definition}

\newtheorem{Rem}[Th]{Remark}

\def\@setemails{%
  \ifnum\theg@author > 1
\mbox{{\itshape E-mail addresses}:\space}{\ttfamily\emails}. \else
\mbox{{\itshape E-mail address}:\space}{\ttfamily\emails}. \fi%
}

\newcommand{\De} {\Delta}
\newcommand{\be} {\begin{equation}}
\newcommand{\ee} {\end{equation}}
\newcommand{\la} {\lambda}
\newcommand{\R} {\mathbb{R}}

\newcommand{\bn}{\mathbb{B}^{N}}

\newcommand{\dx}{\,dV_{\mathbb{B}^N}}

\newcommand{\rn}{\mathbb{R}^{N}}

\newcommand{\intr}{\int_{\rn}}
\newcommand{\intb}{\int_{B^N}}
\newcommand{\inth}{\int_{\bn}}

\newcommand{\calu}{\mathcal{U}}

\newcommand{\authorfootnotes}{\renewcommand\thefootnote{\@fnsymbol\c@footnote}}%

\def\e{{\text{e}}}

\numberwithin{equation}{section} \allowdisplaybreaks

\title[log-perturbed Br\'ezis-Nirenberg problem on the hyperbolic space]{A note on the log-perturbed Br\'ezis-Nirenberg problem on the hyperbolic space}

\author[M. Ghosh]{Monideep Ghosh}
\author[A. Joseph]{Anumol Joseph}
\author[D. Karmakar]{Debabrata Karmakar}
\email{\tt monideep@tifrbng.res.in, anumol24@tifrbng.res.in,  debabrata@tifrbng.res.in}

 \address[]{Tata Institute of Fundamental Research, Centre For Applicable Mathematics}
 \address[]{Post Bag No 6503, GKVK Post Office,
Sharada Nagar, Chikkabommasandra,
Bangalore 560065,
Karnataka, India}



\begin{document}
\pagenumbering{arabic}
\begin{abstract}
We consider the log-perturbed Br\'ezis-Nirenberg problem on the hyperbolic space
 \begin{align*}
 \Delta_{\bn}u+\la u +|u|^{p-1}u+\theta u \ln u^2 =0,  \ \ \ \ u \in H^1(\bn), \ u > 0 \ \mbox{in} \ \bn,
 \end{align*}
 and study the existence vs non-existence results. We show that whenever $\theta >0,$ there exists an $H^1$-solution, while for $\theta <0$, there does not exist a positive solution in a reasonably general class. Since the perturbation $ u \ln u^2$ changes sign, Pohozaev type identities do not yield any non-existence results. The main contribution of this article is obtaining an ``almost" precise lower asymptotic decay estimate on the positive solutions for $\theta <0,$ culminating in proving their non-existence assertion.    
 \end{abstract}

\maketitle
{\small\tableofcontents}
%%%%%%%%%%%%%%%%%%%%%%%%%%%%%%%%%%%%%%%%%%%%%%%%%%%%%%%

\section{Introduction}
We investigate the existence or non-existence of positive solutions to the Br\'ezis-Nirenberg problem with a logarithmic perturbation in the hyperbolic space. The primary focus of this article is to differentiate a critical threshold that separates the existence and non-existence of solutions. While demonstrating the compactness of a constrained minimization problem below a certain energy threshold provides a clear path to positive solutions, establishing the non-existence of solutions does not seem to have a straightforward strategy. Therefore, proving the non-existence of solutions requires a problem-specific approach that demands a more detailed examination of the problem at hand. Additionally, determining an optimal critical threshold that distinguishes between the existence and non-existence of solutions, in our humble opinion, is inherently an interesting problem to explore.

\medskip

In this article, we have obtained a quite clean existence vs non-existence result of the following problem

 \begin{align}\label{LSE} 
 \Delta_{\bn}u+\la u +|u|^{p-1}u+\theta u \ln u^2 =0,  \ \ \ \ u \in H^1(\bn), \ u > 0 \ \mbox{in} \ \bn,
 \end{align}
where  we assume the parameters satisfy 
$$
N \geq 3, \la \in \mathbb{R}, \theta \in \mathbb{R}, \ \mbox{and} \  1 < p   \leq 2^{\star}-1,
$$ 
and $2^\star = \frac{2N}{N-2}$ is the critical exponent in regard to the embedding of $H^1(\bn)$ into $L^{2^{\star}}(\bn).$ When $N=2$, we consider any $p \in (1, \infty).$

\medskip

Here and throughout the article $\bn$ denotes the ball model of the hyperbolic $N$-space,
$\Delta_{\bn}$ denotes the Laplace-Beltrami operator, and $\dx$ is the volume element. Before defining an appropriate  notion of a solution to \eqref{LSE}, let us briefly introduce the necessary terminologies. 

\medskip

Let $\lambda_1$ denotes the $L^2$-bottom of the spectrum of $-\Delta_{\bn}$ defined by 
\begin{align}\label{la1}
\la_1 : = \inf_{u \in C_c^{\infty}(\bn)} \frac{\|\nabla_{\bn} u\|_{2}^2}{\| u\|_{2}^2} = \frac{(N-1)^2}{4},
\end{align}
where $\nabla_{\bn}$ is the gradient vector field and $\|\cdot\|_q$ denotes the $L^q$-norm with respect to the volume element $\dx.$ 

\medskip

Let $H^1(\bn)$ be the classical Sobolev space defined by the closure of $C_c^{\infty}(\bn)$ with respect to the norm $\|u\|_{H^1(\bn)}=\|\nabla_{\bn} u\|_{2}.$  Thanks to \eqref{la1}
the norms
 \begin{align*}
\|u\|_{\la} : = \left(\int_{\bn} \left(|\nabla_{\bn} u|_g^2 - \la u^2\right) \ \dx \right)^{\frac{1}{2}}, 
 \end{align*}
are all equivalent as long as $\la < \la_1.$ When $\la = \la_1,$ we define $\mathcal{H}^1(\bn) := \overline{C_c^{\infty}(\bn)}^{\|\cdot\|_{\la_1}},$ which is a bigger space than $H^1(\bn),$ with strict inclusion -- there exist elements of $\mathcal{H}^1(\bn)$ which are not square integrable. However, one can show that $\mathcal{H}^1(\bn) \subset H^1_{\mbox{\tiny{loc}}}(\bn).$

\medskip

We next define a notion of a (local) solution to \eqref{LSE}.

\medskip

\begin{Def} \label{wf}
We say $u\in H^1_{\mbox{\tiny{loc}}}(\bn)$ is a weak solution of (\ref{LSE}), if  $u$ verifies
\begin{align*}
\inth \langle \nabla_{\bn} u,\nabla_{\bn} \phi \rangle_g \dx-\la\inth u\phi\dx=\theta \inth \phi u\ln u^2\dx +\inth \phi |u|^{p-1}u\dx
\end{align*}
for all $\phi\in C_c^{\infty}(\bn).$ If $u \in H^1(\bn),$ we call it an energy solution.

\medskip

By definition solutions are $H^1_{\mbox{\tiny{loc}}}(\bn)$ and therefore by applying standard (local) elliptic regularity theory, we see that a weak solution, if exists, is always smooth and hence a classical solution.
\end{Def}

The following are the main results of this article.

\begin{Th} \label{main}
Let $N \geq 4,$  $\la \in \mathbb{R},$ and $1 < p \leq 2^{\star} - 1.$  
\begin{itemize}
\item[(a)]  If $\theta > 0,$ then there exists a positive classical solution $u \in H^1(\bn)$ to the equation \eqref{LSE}, which is also a ground state solution. 
\item[(b)] If $\theta < 0,$ then there does not exists any positive solution in $H^1(\bn)$ to \eqref{LSE}. 
\end{itemize}
\end{Th}

In Theorem \ref{main}, a ground state solution means it is a solution with the least energy in an appropriate sense defined in section \ref{prelim}. In addition, we will show that a positive radial energy solution to \eqref{LSE} is strictly decreasing when $\theta>0$.
The proof of the non-existence result relies on a delicate asymptotic decay estimate from below for the positive $H^1$-solutions. With a little bit more work, we can show non-existence result for a larger class $\mathcal{H}^1(\bn).$

\begin{Th} \label{main2}
Let $\la \in \mathbb{R},$ and either $N \geq 3, 1 < p \leq 2^{\star} - 1,$ or $N=2, 1 < p < \infty.$  Assume further that $\theta<0.$
\begin{itemize}
\item[(a)] If $u$ is a positive $\mathcal{H}^1(\bn)$-solution to \eqref{LSE}, then there exist $R_0>0$ and $C_0>0$ such that 
\begin{align*}
C_0 \ sinh \left(\frac{d(0,x)}{2}\right)^{-(N-1)}\leq u(x), \, \ \ \ \mbox{for all} \ x\in \bn \ \mbox{with} \ d(0,x)\geq R_0.
\end{align*}
\item[(b)] There does not exist any positive  $\mathcal{H}^1(\bn)$-solution to \eqref{LSE}.
\end{itemize}
\end{Th}

The expression $d(0,x)$ in Theorem \ref{main2}(a) stands for the hyperbolic distance between $0$ and $x$
(see Section \ref{prelim}).
In addition to the above non-existence results, we show in Section \ref{nonexistence} that even there does not exist a positive solution satisfying a ``reasonable" asymptotic decay at infinity.

\medskip

Before proceeding further, let us first review the precedent related works in the Euclidean and the hyperbolic space. 
 
\medskip

The significant research in this field began with the influential work of Br\'ezis and Nirenberg \cite{BN-1983} in 1983. In their work, they demonstrated that when $\theta = 0$, the problem \eqref{LSE} with $p = 2^{\star} - 1$ on a bounded domain $\Omega$ in $\mathbb{R}^N$ with Dirichlet boundary data admits a positive solution if $\lambda < \lambda_1(\Omega)$, where $\lambda_1(\Omega)$ is the first Dirichlet eigenvalue of $-\Delta$ on $\Omega$. 
The arguments presented by Br\'ezis and Nirenberg had taken inspiration from Aubin's work on Yamabe's problem \cite{AubinY}. Due to the extensive literature in this area, it is  out of our scope to mention all of them. For a further discussion on Yamabe problem and related topics, we refer to the citations \cite{Ya,TruY,AubinY, Sch,U, Taubes1, Taubes2, Struwebook}, subsequent related works and the monographs by Aubin \cite{AubinBook} and by A. Malchiodi \cite{Andrea}.

\medskip

Br\'ezis and Nirenberg \cite{BN-1983} also examined the existence of positive solutions to a perturbed problem. Further related developments have appeared in Adimurthi et al. \cite{Adi} and in Dutta \cite{Ramya}.
 Nevertheless, their assumptions regarding the perturbed problem do not encompass the log-type perturbation considered here due to the sub-linear growth at the origin. 
\medskip

The case where $\theta \neq 0$ has been recently studied by Deng et al. \cite{DPS-2021, DHPZ-2023} and obtained several existence and non-existence results. Regarding the same problem on whole space $\R^N$,  the existence of positive ground state solutions and least energy sign-changing solutions are also affirmative for $\theta >0$ \cite{DPS-2021}.
 
 \medskip
 
One of the key concepts from the work of \cite{BN-1983} demonstrates that the corresponding Euler-Lagrange functional is compact below a certain energy threshold, leading to the existence results. 
However, some hidden complexities arise associated with a log-type perturbation, which we shall now describe. 
First of all, since the associated energy functional corresponding to $\theta\neq 0$ is not $C^1$ when considered as a functional on $H^1(\bn)$ (with appropriate integrability assumptions), we cannot apply the classical theory of critical point directly. An early development in this direction appeared for the study of time-dependent logarithmic Schr\"odinger equation 
$$
\iota\partial_tu + \Delta u + \theta u\ln u^2 = 0
$$
in $\rn.$
There are several remedies in the literature to address this issue. In \cite{CT-1983}, Cazenave worked in a suitable Orlicz space endowed with a Luxemburg type norm to make the functional well-defined and $C^1$ smooth. In \cite{SS-2015}, by applying non-smooth critical point theory for lower semi-continuous functionals, Squassina
and Szulkin studied the following logarithmic Schr\"odinger equation:
\begin{align}\label{surveyeqn1}
-\De u+V(x)u=Q(x)u\ln u^2\quad  \mbox{in} \ \R^N,
\end{align}
where $V(x)$ and $Q(x)$ are spatially periodic. They showed that a positive ground-state solution exists. Moreover, they demonstrated that infinitely many high-energy solutions exist, which are geometrically distinct under $\mathbb{Z}^N$-action. 

\medskip

On the other hand, using a penalization technique, Tanaka and Zhang \cite{TZ-2017} obtained infinitely many
multi-bump geometrically distinct solutions of equation \eqref{surveyeqn1}. The authors first penalized the nonlinearity around the origin, then by considering the spatially $2L$-periodic problems ($L>>1$), proved the existence of infinitely many multi-bump geometrically distinct solutions for the modified equation. Here, we adopt the direct approach of constrained minimization considered by Shuai \cite{S-2019}, who investigated the existence and nonexistence of positive ground state solution, least energy sign-changing solution, and infinitely many nodal solutions for equation \eqref{surveyeqn1} with $Q(x)\equiv 1$ under different types of potentials $V$.  We also refer to the references \cite{DMS-2014,JS-2016,GLN-2010,ZW-2020} for related works.

\medskip

To our knowledge, only the case $\theta = 0$ has been studied in the hyperbolic space. This topic was pioneered by Mancini and Sandeep \cite{MS-2008}, who proved the existence of a positive solution in $\mathcal{H}^1(\bn)$ if 
 \begin{align*}
\mbox{{(\bf H1)}} \ \ \ \ \ \ \ \ \ \ \ \ \ \ \ \ \ \ \
\begin{cases}
\lambda \leq \frac{(N-1)^2}{4}, \ \ \ \mbox{when} \ 1 < p < \frac{N+2}{N-2}, \ \mbox{and} \ N \geq 3, \\ \\
\frac{N(N-2)}{4} < \lambda \leq \frac{(N-1)^2}{4}, \ \ \ \mbox{when} \  p = \frac{N+2}{N-2}, \ \mbox{and} \ N \geq 4, \ \ \ \ \ \ \ \
\end{cases}
\end{align*}
 holds. Moreover, when $\la < \la_1$ the solution is in $H^1(\bn),$ otherwise, it is in $\mathcal{H}^1(\bn) \backslash H^1(\bn)$. In addition, Ganguly and Sandeep \cite{DebdipS} confirmed that for $p = \frac{N+2}{N-2}$ and $\la \leq \frac{N(N-2)}{4},$
 \eqref{LSE} with $\theta = 0$ does not even admit a non-trivial solution. 
 Our main theorem states that when $\theta < 0$, there is no positive solution even in $\mathcal{H}^1(\bn)$,
 irrespective of the values of $\lambda.$ Regarding the sub-critical case, the authors of \cite{Bonf12} discussed the classification of radial solutions (not necessarily finite energy) and their qualitative behavior such as positivity, number of zeroes and asymptotic behavior at infinity in terms of the initial value. See also \cite{San_a, San_b, San_c, Bhakta} for related works on $\bn$ corresponding to $\theta = 0.$

 \medskip
 
 Before concluding the introduction let us remark that the log-type perturbation is not merely a technical hypothesis, it has physical meaning as well. For example, the time dependent logarithmic Schrodinger equation
 \begin{align}\label{TimeLSE}
 \iota\frac{\partial\psi}{\partial t}=D\Delta \psi+\sigma\ln(|\psi|^2)\psi,
 \end{align} 
 where D being the diffusion constant and
$\sigma \in\R \setminus \{0\}$ representing the strength of the (attractive or repulsive) nonlinear interaction, finds its applications to quantum mechanics, quantum optics, nuclear physics, transport and diffusion phenomena, open quantum systems, effective quantum gravity, theory of superfluidity and Bose-Einstein condensation. See \cite{Z-2010} and the references therein for physical motivation. Various meaningful physical interpretations have been given to the presence of the logarithmic potential in
the Schrödinger equation. Indeed, it can be understood as the effect of statistical uncertainty or as the potential energy
associated with the information encoded in the matter distribution described by the probability density $|\psi(t, x)|^2$. Recently, equation \eqref{TimeLSE} has proved useful for the modelling of several nonlinear phenomena including capillary fluids \cite{ML-2003} and geophysical
applications of magma transport \cite{Martino2003}, as well as nuclear physics \cite{H-1985}, Brownian dynamics or photochemistry. Besides, one
of its most relevant potential applications nowadays seems to concern the modelling of quantum dissipative interactions
between a particle ensemble and a thermal reservoir of phonons when a Fokker–Planck scattering mechanism comes into
play (see \cite{L-2004,L-2009}).

\medskip

The outline of this article is as follows. In Section \ref{prelim} we introduce and recall the necessary tools and terminologies. 
Earlier, we mentioned that if the constrained energy level is strictly less than a certain threshold, it leads to the compactness of the minimizing sequence and consequently leads to a solution. However, estimating the value of energy brings difficulties, especially for $N=4$, where the Aubin-Talenti bubbles (the Euclidean Sobolev extremizers) are not square integrable. We carry out these estimates in Section \ref{critical value}. In Section \ref{existence}, we prove the existence of positive ground states for $\theta >0.$

\medskip

The main contribution of this article leans on the non-existence results for $\theta <0$. Due to the sign-changing behavior of $u\ln u^2$, Derrick-Pohozaev's identity \cite{Poho} does not provide satisfactory results for this case. The first eigenfunction method as Deng et.al.\cite{DHPZ-2023} have implemented in bounded domains of $\mathbb{R}^N$,  is also not applicable in this context because the first eigenfunction of the Laplace-Beltrami operator $-\Delta_{\bn}$ is not square integrable. We can overcome this challenge by deriving a lower asymptotic decay of the solutions in the regime of $\theta<0$ and $\lambda\in \mathbb{R}$. 

\medskip

Roughly the idea is as follows: clubbing the terms $(\la + \theta \ln u^2)u$ and treating it as a linear term, we 
could speculate from the work of \cite{MS-2008} that a positive $H^1$-solution should behave like $(1 - |x|^2)^{\frac{N-1}{2}}$ at infinity. However, making this precise brings additional difficulties. A natural approach would be to construct a suitable barrier. Since the given solution $u$ is a supersolution to $-\Delta_{\bn} u - \la_0 u \geq 0$ for any $\la_0,$ outside a large ball, all we need is a sub-solution $v \in H^1(\bn)$  to $-\Delta_{\bn} v - \la_0 v \leq 0$ satisfying the necessary decay assumption. Unfortunately, such a sub-solution exists only if $\la_0 > \la_1$ and hence the comparison principle fails for such operator $-\Delta_{\bn} - \la_0 .$ Nevertheless, we were able to circumvent this difficulty and prove the desired lower asymptotic decay of positive solutions. 

\medskip

In addition, using suitable interaction estimates, we have demonstrated that there is no positive classical solution with a reasonable asymptotic decay. Section \ref{nonexistence} is devoted to all the non-existence results obtained in this article.
 
\section{Preliminaries}\label{prelim}

\subsection{Notations} Throughout the article, we write $A \lesssim B$ to mean that there exists a constant C (depending on the natural parameters $N,p,\la,\theta$) such that $A \leq CB.$ $A \gtrsim B$ is similarly defined. We write $A \approx B$ if both $A \lesssim B$ and $A \gtrsim B$ hold. If we write $A \lesssim_{\delta} B,$ then this would mean the constant $C$ also depends on $\delta.$ The notation $\psi^+$ denotes the positive part $\max\{\psi,0\}.$ Similarly, $\psi^- = \max\{-\psi, 0\}.$

\subsection{The ball model} We briefly introduce the necessary concepts and refer to \cite{Ratcliffe} for more details. 
The Euclidean unit ball $B^N:= \{x \in \mathbb{R}^N: |x|^2<1\}$ equipped with the Riemannian metric
\begin{align*}
{\rm d}s^2 = \left(\frac{2}{1-|x|^2}\right)^2 \, {\rm d}x^2
\end{align*}
constitute the ball model for the hyperbolic $N$-space, where ${\rm d}x^2$ is the standard Euclidean metric and $|x|^2 = \sum_{i=1}^Nx_i^2$ is the standard Euclidean length. The volume element $\dx$ 
is given by  $\left(\frac{2}{1-|x|^2}\right)^N \ dx,$ $dx$ being the Lebesgue measure.

\medskip

 The hyperbolic distance between two points $x$ and $y$ in $\bn$ will be denoted by $d(x, y).$ The  distance between $x$ and the origin can be computed explicitly  by the formula
\begin{align*}
\rho := \, d(x, 0) = \int_{0}^{|x|} \frac{2}{1 - s^2} \, {\rm d}s \, = \, \log \left(\frac{1 + |x|}{1 - |x|}\right),
\end{align*}
and therefore  $|x| = \tanh \frac{\rho}{2}.$ More generally, one can compute the hyperbolic distance between any two points $x, y \in \bn$ and it is given by 
\begin{align*}
\cosh d(x, y) =  \left( 1 + \dfrac{2|x - y|^2}{(1 - |x|^2)(1 - |y|^2)} \right), \ \mbox{or}, \
\sinh \left(\frac{d(x,y)}{2} \right)= \frac{|x-y|}{\sqrt{(1-|x|^2)(1-|y|^2)}}.
\end{align*}
Geodesic balls in $\bn$ of radius $r$ centred at $x \in \bn$ will be denoted by 
$$
B_r(x) : = \{ y \in \bn : d(x, y) < r \},
$$
  or simply by $B_r$ if $x=0.$
  
 For $b \in \mathbb{B}^N,$ the hyperbolic translation $\tau_{b}: \mathbb{B}^N \rightarrow \mathbb{B}^N$ that takes $0$ to $b$ is defined by  the following formula
  \begin{align} \label{hyperbolictranslation}
  \tau_{b}(x) := \frac{(1 - |b|^2)x + (|x|^2 + 2 x \cdot b + 1)b}{|b|^2|x|^2 + 2 x\cdot b  + 1}.
 \end{align}
 It turns out that $\tau_{b}$ is an isometry, and together with the orthogonal transformations they form the M\"obius group of $B^N$ (see \cite{Ratcliffe}, Theorem 4.4.6 for details and further discussions on isometries). 

\subsection{Framework}

The solutions of \eqref{LSE} are the critical points of
\begin{align*}
J_p(u)=\frac{1}{2}\inth |\nabla_{\bn}u|_g^2\dx-\frac{\la}{2}\inth u^2\dx -\frac{1}{p+1}\|u\|^{p+1}_{p+1}-\frac{\theta}{2}\inth u^2(\ln u^2-1)\dx
\end{align*}
defined on appropriate function space defined below.
We define
\begin{align*}
H_r^1(\bn)=\{ u \in H^1(\bn)| \ u \text{ is radial}\}.
\end{align*}

Because of the infinite volume, 
the log term in the expression of $J_p$ does not make sense in $H^1(\bn).$ Hence we need to introduce the following subspace of $H^1(\bn)$ 
\begin{align*}
X=\{u\in H^1_r(\bn)\setminus \{0\} \ | \ u^2\ln u^2\in L^1(\bn)\}.
\end{align*}
Clearly $X$  is dense in $H^1_r(\bn),$ since $H^1_r(\bn)\cap C_c^\infty(\bn)\setminus\{0\}$ is contained in $X$. The existence of a positive solution will be obtained by constrained minimization on the Nehari set $\mathcal{N}_p =\{u \in X \ | \ I_p(u)=0\}$
where 
\begin{align*}
I_p(u)= \inth |\nabla_{\bn}u|_g^2\dx -\la\inth u^2\dx-\inth |u|^{p+1}\dx -\theta \inth u^2 \ln u^2\dx .
\end{align*}

We denote the critical value
\begin{align*}
d_p= \inf_{u\in \mathcal{N}_p} J_p(u).
\end{align*} 
The natural plan is to show $d_p$ is attained and the minimizer is a solution in the sense of Definition \ref{wf}. Moreover, thanks to the integrability of $u^2\ln u^2,$ the weak formulation holds for all $\phi \in X.$ It is worth mentioning that neither the space $X$ is a Banach space with respect to the $H^1$-norm, nor the functional $J_p$ is of class $C^1,$ wherever they are defined.

 \subsection{Basic Inequalities}
 For the convenience of the reader, we gather well-known inequalities required in this article in the next two subsections.
 
 \medskip
 
 \noindent
{\bf $\bullet$ Sobolev inequality in $\rn$.} Let $N \geq 3$. There exists a best constant $S = S(\rn)$ such that
\begin{align}\label{Sobolev}
S\bigg(\intr u^{2^\star} \ dx \bigg)^\frac{2}{2^\star}\leq \intr |\nabla u|^2 \ dx,
\end{align}
holds for all $u\in C_c^{\infty}(\rn),$ where $2^{\star}=\frac{2N}{N-2}$ is called the critical Sobolev exponent. By density argument, the inequality \eqref{Sobolev} continues to hold for all $u$  
satisfying $\|\nabla u\|_{L^2(\R^n)} <\infty,$ and $\mathcal{L}^n(\{|u| >t\}) < \infty$ for every $t>0,$ where $\|\cdot \|_{L^2(\R^n)}$ denotes the $L^2$-norm and $\mathcal{L}^n$ denotes the Lebesgue measure on $\R^n.$ The explicit value of $S$ is known \cite{Ro} and the equality cases in \eqref{Sobolev} are classified and given by Aubin-Talenti bubbles \cite{Aubin, Talenti}
$$
U[z,\mu] (x)=[N(N-2)]^{\frac{N-2}{4}}\mu^{\frac{N-2}{2}}\left(\frac{1}{1+\mu^2|x - z|^2}\right)^{\frac{N-2}{2}}, \ \  z \in \rn, \mu>0.
$$

\medskip

\noindent
{\bf $\bullet$ The Poincar\'e-Sobolev inequality.} Let $N \geq 3$ and $\lambda \leq \frac{(N-1)^2}{4}$ and $1 < p \leq 2^{\star} -1.$ Then  there exists a best constant $S_{\lambda,p}:=S_{\lambda,p}(\bn)>0$ such that 
\begin{align}\label{S-inq}
S_{\lambda,p}\left(~\int \limits_{\bn}|u|^{p+1} \, \dx \right)^{\frac{2}{p+1}}\leq\int \limits_{\bn}\left(|\nabla_{\bn}u|^{2}-\lambda u^{2}\right)\, \dx
\end{align}
holds for all  $u\in C_c^{\infty}(\bn).$ For $N=2$, the inequality holds for any $p>1$.

By density, \eqref{S-inq} continues to hold for every $u$ belonging to the the closure of $C_c^{\infty}(\bn)$ with respect to the norm $\|u \|_{\lambda}.$

\medskip

The inequality \eqref{S-inq} was proved by Mancini and Sandeep in \cite{MS-2008} and in the same article, they also proved the existence of optimizers under appropriate assumptions on $N,\lambda$ and $p.$ In particular, they showed that under the hypothesis {\bf (H1)}, there always exists a strictly positive, radially symmetric and decreasing extremizer $\calu$
in $H^1(\bn)$ or in $\mathcal{H}^1(\bn),$ depending on the values of $\la.$ It is straightforward to verify that subject to an appropriate normalization 
the obtained extremizer is a positive solution to 
\begin{align}\label{eq1}
 -\Delta_{\bn}u-\lambda u=|u|^{p-1}u~ \ \ \ \  u \in H^{1}(\bn) \ \mbox{or} \ \mathcal{H}^1(\bn).
 \end{align}
 The equation \eqref{eq1} as well as the inequality \eqref{S-inq} is invariant under the conformal group of the ball model, which in this case coincides with the isometry group of the ball model and is generated by the hyperbolic translations $\tau_b, b\in \bn$ and orthogonal transformations. In \cite{MS-2008} Mancini and Sandeep also classified the positive solutions of \eqref{eq1}, which in turn provides the classification of the extremizers of \eqref{S-inq}. Their results are as follows: Under the assumptions {\bf (H1)} with $\la < \la_1$ the set 
 \begin{align*}
\mathcal{Z}_{0}:=\{\calu[b] := \mathcal{U}\circ\tau_{b}: b \in \bn\}
\end{align*}
consists of all the positive solutions to \eqref{eq1} and $c\mathcal{Z}_0, \ c \in \mathbb{R}\backslash \{0\}$ consists of all the nontrivial extremizers of \eqref{S-inq}. 

\medskip

\noindent
{\bf $\bullet$ The log-Sobolev inequality on $\bn$.}
 Let $N \geq 2.$ There exist constants $C_1$ and $C_2$ (depending only on $N$) such that for every $\epsilon>0$  the inequality
\begin{align*}
\inth u^2 \ln u^2\dx \leq \frac{\epsilon}{\pi}\|\nabla_{\bn} u\|_2^2+\|u\|_2^2\left(\ln\|u\|_2^2+C_1-C_2\ln\epsilon \right),
\end{align*}
holds for every $u\in H^1(\bn).$
\begin{proof}
First we assume $\|u\|^2=1$.
Since logarithm is a concave function, by Jenson's inequality, we have
\begin{align}\label{Jensen's Inequality}
\inth \ln (u^2) u^2\dx &= \frac{2}{p-1}\inth \ln (u^{p-1}) u^2\dx \nonumber \\
&\quad \leq \frac{2}{p-1}\ln \inth  u^{p+1}\dx.
\end{align}

Now we apply the  Poincar\'e-Sobolev inequality for particular values of $p$ in the two cases of $N=2$ and $N\geq 3$:

\medskip

Case 1: If $N=2,$ we choose $p=3$, then from \eqref{Jensen's Inequality} and \eqref{S-inq} we have 
\begin{align*}
\inth \ln (u^{2}) u^2\dx &\leq \ln \inth u^4 \dx \\
&\quad \leq 2\ln \left(S_{0,3}^{-1}\inth |\nabla_{\bn} u|_g^2\dx \right).
\end{align*}

\medskip

Case 2: If $N\geq 3$ we take $p=2^{\star}-1$, then from \eqref{Jensen's Inequality} and \eqref{S-inq} we have 
\begin{align*}
\inth \ln (u^{2}) u^2\dx &\leq \frac{2}{2^{\star}-2}\ln \inth u^{2^\star} \dx\\
&\quad\leq \frac{2^{\star}}{2^\star-2} \ln \left(S^{-1} \inth |\nabla_{\bn} u|_g^2 \dx\right)
\\ &\quad =\frac{N}{2} \ln \left(S^{-1} \inth |\nabla_{\bn} u|_g^2 \dx\right).
\end{align*}
Now, using $\ln (ax)\leq \epsilon x +\ln (a\epsilon^{-1})$, we get,
\begin{align}\label{Normalized log-sobolev}
\inth u^2 \ln u^2 \dx &\leq  \tilde{C_1} \ln\bigg( \tilde{C_2} \inth |\nabla_{\bn} u|_g^2\dx\bigg)\nonumber\\
&\leq \frac{\epsilon}{\pi} \|\nabla_{\bn} u\|_2^2+ \tilde{C_1}\ln\bigg(\tilde{C_2}\tilde{C_1}\pi\epsilon^{-1}\bigg)\nonumber\\
&= \frac{\epsilon}{\pi} \|\nabla_{\bn} u\|_2^2+ \bigg(C_1+C_2\ln\epsilon^{-1}\bigg).
\end{align}
The general case follows from \eqref{Normalized log-sobolev} by considering $\frac{u}{\|u\|_2^2}$ instead of $u$. 
\end{proof}

 \subsection{Preliminary Results}
 We now state a few intermediate lemmas required for the proof of existence results.
 
\begin{Lem} \label{ex}
Let $p\in(1,2^{\star}-1], \theta \geq 0$ and assume that $u\in \mathcal{N}_p$ and  $J_p(u)=d_p >0.$ Then $u$ is a positive solution to \eqref{LSE}.
\end{Lem}

Thanks to the above lemma, we now only require to prove that $d_p$ is achieved. The proof of the lemma follows exactly as in Shuai \cite[Theorem 1.1]{S-2019}. For the convenience of the reader we include the details at the end of this section. 
We need a few technical lemmas for the proof of Lemma \ref{ex} and for subsequent uses.
 \begin{Lem}
 The followings hold:
 \begin{itemize}
 \item[(a)] $H^1_r(\bn)$ is compactly embedded in $L^q(\bn)$ for $q\in(2,2^{\star})$.
 \item[(b)]Let $u_n\in H^1_r(\bn)$ such that $\{u_n\}$ is bounded in $H^1_r(\bn)$ and $\{u_n\}$ is bounded in $L^s(\bn)$ for $s>2^{\star}$. Then there exists $u\in H^1_r(\bn)$ such that up to a subsequence $u_n\to u $ in $L^{2^{\star}}(\bn)$.
 \end{itemize}
\end{Lem}

The first one is quite standard, see for example \cite[Theorem 3.1]{BS-2012} for a proof.  $(b)$ follows from $(a)$ and interpolation inequality.

 \begin{Lem}
Given $\theta>0$ and $\la\in \mathbb{R}.$ The functional $$\mathcal{F}(u)=-\la\inth u^2\dx -\theta\inth u^2 \ln u^2 \dx$$ is weakly lower semicontinuous on $H^1_r(\bn)$.
 \end{Lem}
 \begin{proof}
  Let $u_n \rightharpoonup u$ in $H^1_r(\bn)$. Then up to a subsequence 
   $u_n\rightarrow u$ in $L^q(\bn), q \in (2,2^{\star}-1),$
and a.e. in $\bn.$ Note that $(u^2\ln u^2)^+\leq u^q$ for $q>2$. Moreover, there exists a small $\delta>0$ such that,
$-\theta u^2\ln u^2-\la u^2\geq 0$ whenever $|u|<\delta$ and for $q>2$, 
$ u^2< \delta^{2-q} |u|^q$ whenever $|u|\geq\delta$. 

  Hence by generalised dominated convergence theorem,
  \begin{align*}
  -\theta\inth (u_n^2\ln u_n^2)^+\dx&\rightarrow -\theta\inth(u^2 \ln u^2)^+\dx,\\
  -\la\int_{\{u_n>\delta\}} u_n^2\dx&\rightarrow -\la\int_{\{u>\delta\}}u^2\dx, 
  \end{align*}
   and by Fatou's Lemma,
  \begin{align*}
  \int_{\{u<\delta\}}(\theta (u^2\ln u^2)^- -\la u^2)\dx \leq \liminf\int_{\{u_n<\delta\}} (\theta (u_n^2\ln u_n^2)^- -\la u_n^2)\dx,\\
    \int_{\{u\geq \delta\}}\theta (u^2\ln u^2)^- \dx \leq \liminf\int_{\{u_n\geq\delta\}} \theta (u_n^2\ln u_n^2)^-\dx.
  \end{align*}
  Adding all the integrals we conclude the proof.
 \end{proof}
\begin{Cor}
For $\theta >0$ and $\la \in \mathbb{R},$ the functionals $J_p(u),\  
I_p(u)$ are weakly lower semicontinuous on $H_r^1(\bn)$ whenever $p\in(1,2^{\star}-1).$ Under the same assumptions $J_{2^\star-1}(u)$, $I_{2^\star-1}(u)$ are lower semicontinuous on $H^1_r(\bn).$ 
\end{Cor}

\begin{Lem}\label{Uniqueness}
Let $\theta>0$ and $u\in X$. Then there exists a unique $t_0>0$ such that $I_p(t_0 u)=0$ for $p\in(1,2^{\star}-1]$.
\end{Lem}
\begin{proof}
For $t>0$,  $I_p(tu)=0$ is equivalent to 
  \begin{align*}
  \|\nabla_{\bn}u\|_{2}^2 - \la \|u\|_2^2&=\frac{1}{t}\bigg(2 t\ln t \bigg(\theta\inth u^2\dx\bigg)+t^p\inth u^{p+1}\dx+\theta t\inth u^2\ln u^2\dx\bigg).
  \end{align*}
  Equivalently,
  \begin{align*}
  \|\nabla_{\bn}u\|_{2}^2 - \la \|u\|_2^2=2\theta \ln t\inth u^2 \dx+ t^{p-1}\inth u^{p+1}\dx+\theta \inth u^2\ln u^2\dx.
  \end{align*}
 The R.H.S. term is strictly increasing in $t$ for $t>0$ whereas the L.H.S. is a constant  and hence there exists a unique $t_0$ such that the above equality holds. 
\end{proof}

\medskip

\noindent
{\bf Proof of Lemma \ref{ex}.}

\begin{proof}

First by symmetrization, we note that $d_p=\displaystyle\inf_{u\in W, \ I_p(u)=0}J_p(u)$ where $W=\{u\in H^1(\bn) \ | \ u\ln u^2\in L^1(\bn)\}.$ Indeed, let $u^{\star}$ be the symmetric decreasing rearrangement of $u.$ Then $I_p(u^{\star})\leq 0.$ By Lemma \ref{Uniqueness} there exists $t \in (0,1]$ such that $I_p(tu^{\star})= 0.$ Then 
\begin{align*}
J_p(tu^{\star}) = J_p(tu^{\star}) - \frac{1}{2}I_p(tu^{\star})&=\frac{t^2\theta}{2}\|u^{\star}\|_{2}^2 + \frac{t^{p+1}}{p+1}\|u^{\star}\|_{p+1}^{p+1}\\
& \leq \frac{1}{2}\|u\|_{2}^2 + \frac{1}{p+1}\|u\|_{p+1}^{p+1} = J_p(u).
\end{align*}

Suppose $u \in W$ be a minimizer. Assume that $J_p^{\prime}(u) \neq 0,$ and let $\phi\in C_c^{\infty}(\bn)$ such that $J_p'(u)\phi\leq-1,$ where  $J_p^{\prime}(u) \phi$ needs to be understood in the sense of Definition \ref{wf}.
Now observe that for the above fixed $\phi\in C_c^{\infty}(\bn)$ and  for all $t >0,\sigma \in \R $, we have
\begin{align*}
& \ \ \ \ J_p'(tu+\sigma\phi)\phi \\
&\leq J_p'(u)\phi +|t-1||\langle u,\phi\rangle_{\la}|+|\sigma| \|\phi\|_{\la}^2+|p\sigma|\inth (|tu|^{p-1}+|\sigma \phi|^{p-1})\phi^2 \dx \\ 
&\ \ \ + |t^p-1| \inth |u|^{p} |\phi|\dx + 2|\sigma|\inth \phi |\ln(|u|+|\phi|)|(|u|+|\phi|\dx\\
&\ \ \ +2|t-1|\inth|\phi u\ln u|\dx + |t\ln t|\inth |u \phi|\dx.
\end{align*}
Hence, there exists $\epsilon>0$ such that whenever $|t-1|<\epsilon$, $|\sigma|\leq \epsilon$ we have
\begin{align}\label{111}
J_p'(t u+\sigma \phi)\phi \leq J_p'(u)\phi +\frac{1}{2}\leq-\frac{1}{2}.
\end{align}
Now define $\eta \in C_c^{\infty}(\bn)$ such that $0\leq \eta \leq 1$, and
\begin{align*}
\eta(t)= 
\begin{cases}
1 &\quad|t-1|\leq \frac{\epsilon}{2}\\
0 &\quad|t-1|>\epsilon.
\end{cases}
\end{align*}

Set $g(t)=I_p(tu+\epsilon\eta(t)\phi)$ for  $t>0$.
By Lemma \ref{Uniqueness} we know that $I_p(tu)>0$ if $0<t<1$ and $I_p(tu)<0$ if $t>1$. 
Therefore $g(1-\epsilon)>0$ and $g(1+\epsilon)<0$.  
By continuity of $g$, there exists a $t_0\in(1-\epsilon,1+\epsilon)$ such that $g(t_0)=0$. Unwrapping the definition of $g$, we get
\begin{align*}
I_p(t_0u+\epsilon \eta(t_0)\phi)=0.
\end{align*}
Therefore $(t_0 u+\eta(t_0)\phi)\in  \mathcal{N}_p.$ The same argument yields $J_p(t_0u)\leq J_p(u)$ and hence by \eqref{111}, $J_p(t_0u+\epsilon\eta(t_0)\phi)=J_p(t_0u)+\eta(t_0)\epsilon \int_0^1\langle J'_p(t_0u+s\eta(t_0)\epsilon\phi),\phi\rangle ds\leq J_p(t_0u)-\eta(t_0)\frac{\epsilon}{2}.$ As a result we have
\begin{align}\label{strict inequality}
 J_p(t_0u+\epsilon\eta(t_0) \phi)\leq J_p(t_0u)-\frac{\eta(t_0) \epsilon}{2}< J_p(u)
 \end{align}
 which is a contradiction. That $u\geq 0$ is a standard argument as $d_p>0$. 
 The strict positivity follows from the maximum principle \cite{V-1984}. This completes the proof.
\end{proof}

\section{Estimation of $d_{2^\star-1}$}\label{critical value}

In this section, we show that $d_{2^\star-1}<\frac{1}{N}S^{\frac{N}{2}}$, where $S$ is the best constant in the classical Sobolev inequality in $\rn$.
The basic idea goes back to Br\'ezis and Nirenberg \cite{BN-1983} followed by the recent work of Deng et al. \cite{DHPZ-2023} incorporating the log term. We look for some suitable $v_\epsilon$ such that $\sup_{t\geq0}J(t v_{\epsilon})<\frac{1}{N}S^{\frac{N}{2}}$.  

\medskip

The extremizers of the classical Sobolev inequality in $\rn$, called the Aubin-Talenti bubbles $U (x)=[N(N-2)]^{\frac{N-2}{4}}\left(\frac{1}{1+|x|^2}\right)^{\frac{N-2}{2}}$ provides the a suitable candidate for this purpose. We define an appropriate dilation of $U,$ $U_{\epsilon}(x) :=[N(N-2)]^{\frac{N-2}{4}}\left(\frac{\epsilon}{\epsilon^2+|x|^2}\right)^{\frac{N-2}{2}}$  such that $\|U_{\epsilon}\|_{2^{\star}}^{2^{\star}}=\|\nabla U_{\epsilon}\|_2^2=S^{\frac{N}{2}}.$
Let $\phi\in C_c^{\infty}(\rn)$ be a radial cut-off function satisfying 
$\phi(x)=1 \quad \text{for}\quad 0\leq |x|\leq \rho, \
0\leq \phi(x)\leq 1\quad  \text{for} \quad \rho\leq |x|\leq 2\rho, \
\phi(x)=0 \quad \text{for}\quad |x|>2\rho
$
for some fixed $\rho>0$ small. Define
$v_\epsilon=\phi U_{\epsilon} $.
We recall the following two results  (see \cite{BN-1983, W-1996, DHPZ-2023}). 
\begin{Lem}\label{BNs}
	If $N\geq 4$, then we have, as $\epsilon\to 0^+$,
	\begin{align}\label{est grad,L^2^*}
		\intb |\nabla v_{\epsilon}|^2 = S^{\frac{N}{2}}+O(\epsilon^{N-2}),\\
		\intb |v_{\epsilon}|^{2^{\star}}= S^{\frac{N}{2}}+O(\epsilon^{N}),
	\end{align}
	and
	\begin{align}\label{L^2}
		\intb |v_{\epsilon}|^2=
		\begin{cases}
			d\epsilon^2 |\ln\epsilon|+ O(\epsilon^2), \quad &\text{if } N=4,\\
			d\epsilon^2+O(\epsilon^{N-2}),\quad &\text{if } N\geq 5,
		\end{cases}
	\end{align}
	where $d$ is a positive constant.
\end{Lem}

\begin{Lem}[\cite{DHPZ-2023}]\label{est vlnv}
	As $\epsilon\to 0^+$, we have
	\begin{align}\label{est vlnv n>=5}
		\intb v_{\epsilon}^2\ln v_{\epsilon}^2=C_0\epsilon^2|\ln\epsilon|+O(\epsilon^2) \quad \text{for } N\geq 5,
	\end{align}
	and for $N=4$,
	\begin{align}\label{est vlnv n=4}
		\begin{cases}
			\intb v_{\epsilon}^2\ln v_{\epsilon}^2 \geq 8\ln\bigg(\frac{8(\epsilon^2+\rho^2)}{\e(\epsilon^2+4\rho^2)^2}\bigg)\omega_4\epsilon^2 |\ln \epsilon|+O(\epsilon^2),\\
			\intb v_{\epsilon}^2\ln v_{\epsilon}^2 \leq 8\ln\bigg(\frac{8\e(\epsilon^2+4\rho^2)}{(\epsilon^2+\rho^2)^2}\bigg)\omega_4\epsilon^2|\ln{\epsilon}| +O(\epsilon^2), 
		\end{cases}
	\end{align}
	where $C_0$ is a positive constant and $\omega_4$ denotes the area of unit sphere in $\mathbb{R}^4$. In particular, for $N=4,$ the co-efficient of $\epsilon^2|\ln{\epsilon}|$ can be made as large as possible by choosing $\rho \approx \epsilon \rightarrow 0^+.$
\end{Lem}

In order to implement the above estimates, we need to make a conformal change of metric. For $u \in H^1(\bn),$ we set $v=\bigg(\frac{2}{1-|x|^2}\bigg)^{\frac{N}{2}-1}u$.
Then we have
\begin{align*}
	J_{2^\star-1}(u)&=\frac{1}{2}\inth (|\nabla_{\bn} u|_g^2-\la u^2)\dx-\frac{1}{2^{\star}}\inth|u|^{2^{\star}}\dx -\frac{\theta}{2}\inth u^2(\ln u^2-1)\dx\\
	&\quad = \frac{1}{2}\intb |\nabla v|^2dx-\frac{1}{2}\intb  g v^2 dx
	-\frac{1}{2^{\star}}\intb |v|^{2^\star}dx -\frac{\theta}{2} \intb h v^2\ln (v^2) dx \\
	&\quad=:\tilde{J}(v).
\end{align*}
\begin{align*}
	I_{2^\star-1}(u)&=\inth (|\nabla_{\bn} u|_g^2-\la u^2)\dx-\inth|u|^{2^{\star}}\dx -{\theta}\inth u^2(\ln u^2-1)\dx\\
	&\quad = \intb |\nabla v|^2dx-\intb  (g+\theta h) v^2 dx
	-\frac{1}{2^{\star}}\intb |v|^{2^\star}dx -\theta \intb h v^2\ln (v^2) dx \\
	&\quad=:\tilde{I}(v),
\end{align*}
where $g(x)=\bigg(\la-\frac{N(N-2)}{4}-\theta -(N-2)\theta \ln\bigg(\frac{2}{1-|x|^2}\bigg)\bigg)\bigg(\frac{2}{1-|x|^2}\bigg)^2$ and $h(x)=\bigg(\frac{2}{1-|x|^2}\bigg)^2$.

Let $v_{\epsilon}$ be as in Lemma \ref{est vlnv}. Then, we have 
\begin{Lem}
	As $\epsilon\to 0^+$ we have
	\begin{align*}
		\intb h(x) v_{\epsilon}^2(x)\ln v_{\epsilon}^2(x) dx= 4\intb v_{\epsilon}^2(x)\ln v_{\epsilon}^2(x)dx+O(\epsilon^2)\quad \mbox{for}\ \ \ N\geq 5,
	\end{align*}
	and for $N=4$ 
	\begin{align*}
		\intb h(x) v_{\epsilon}^2(x)\ln v_{\epsilon}^2(x) dx= 4\intb v_{\epsilon}^2(x)\ln v_{\epsilon}^2(x)dx+c_{\rho,\epsilon}\epsilon^2 |\ln\epsilon| +o(\epsilon^2|\ln\epsilon|),
	\end{align*}
	where $|c_{\rho,\epsilon}|\lesssim 1,$ a dimensional constant, whenever $\rho\leq \frac{1}{4}$.
\end{Lem}
\begin{proof}
	We only consider the case $N = 4,$ as $N \geq 5$ is much more simpler because of the $L^p(\bn)$ integrability of $U$ for all $p>\frac{5}{3}$.
	We use two basic integrals
	\begin{align}\label{Integral1}
		\int_0^{\frac{\rho}{\epsilon}}\frac{r^5}{(1+r^2)^2} dr\approx \left(\frac{\rho}{\epsilon}\right)^2, \ \ \ 
		\int_0^{\frac{2\rho}{\epsilon}}\frac{r^5}{(1+r^2)^2}\ln(1+r^2) dr \approx \rho^2\left(\frac{1}{\epsilon^2}\ln\frac{1}{\epsilon} \right).
	\end{align}
	Now we will estimate $\intb (h(x)-4)v_{\epsilon}^2(x)\ln v_{\epsilon}^2(x).$ We decompose the integral into three parts $\Romannum{1}-\Romannum{3},$
	and estimate each of the integrals one by one.
	\begin{align*}
		\Romannum{3}=8\epsilon^4\int_{B_{\frac{1}{\epsilon}}\setminus B_{\frac{\rho}{\epsilon}}}\frac{|x|^2}{(1+|x|^2)^2}\frac{2-\epsilon^2|x|^2}{(1-\epsilon^2|x|^2)^2}\phi^2(\epsilon x)\ln 8\phi(\epsilon x)dx
		=O(\epsilon^2).
	\end{align*}
	\begin{align*}
		\Romannum{2}=8\ln8\omega_4 \epsilon^4 \int_0^\frac{\rho}{\epsilon}\frac{r^5}{(1+r^2)^2}\frac{2-\epsilon^2r^2}{(1-\epsilon^2 r^2)^2} dr
		\approx  \int_0^{\frac{\rho}{\epsilon}}\frac{r^5}{(1+r^2)^2}dr
		\approx \rho^2 \epsilon^2 +o(\epsilon^2),
	\end{align*}
	where the constant in $\approx$ is bounded below and above by $2-\rho$ and $\frac{2}{(1-\rho)^2}$ respectively (up to a universal constant). In the same spirit
	\begin{align*}
		\Romannum{1}&=8\omega_4\epsilon^4\int_0^{\frac{2\rho}{\epsilon}}\frac{r^5}{(1+r^2)^2}\phi(\epsilon r)\frac{2-\epsilon r^2}{(1-\epsilon^2r^2)^2}\ln(\frac{1}{\epsilon^2(1+r^2)^2})dr\\
		&=8\omega_4\epsilon^4\ln\left(\frac{1}{\epsilon}\right)\int_0^{2\frac{\rho}{\epsilon}}\frac{r^5}{(1+r^2)^2}\phi(\epsilon r )\frac{2-\epsilon^2r^2}{(1-\epsilon^2r^2)^2}dr \\
		& \ \ \ \ \ \ \ +\omega_4\int_0^{2\frac{\rho}{\epsilon}}\frac{r^5}{(1+r^2)^2}\phi(\epsilon r )\frac{2-\epsilon^2r^2}{(1-\epsilon^2r^2)^2}\ln(\frac{1}{(1+r^2)^2})^2dr\\
		&=\Romannum{1}_1+\Romannum{1}_2.
	\end{align*}
	Also note that
	\begin{align*}
		\Romannum{1}_1\approx \epsilon^4\ln\left(\frac{1}{\epsilon}\right)\int_0^{2\frac{\rho}{\epsilon}}\frac{r^5}{(1+r^2)^2}dr
		\approx \rho^2\left(\epsilon^2\ln\frac{1}{\epsilon}\right)+o\left(\epsilon^2\ln\frac{1}{\epsilon}\right),
	\end{align*}
	and
	\begin{align*}
		\Romannum{1}_2 \approx -\epsilon^4\int_0^{2\frac{\rho}{\epsilon}}\frac{r^5}{(1+r^2)^2}ln(1+r^2)dr
		&\approx -\rho^2\left(\epsilon^2\ln\frac{1}{\epsilon}\right)+o\left(\epsilon^2\ln\frac{1}{\epsilon}\right)\\
		&\approx -\epsilon^2\ln\left(\frac{1}{\epsilon}\right)+o\left(\epsilon^2\ln\frac{1}{\epsilon}\right),
	\end{align*}
	where the constants in $\approx$ are bounded and lie, up to a universal constant times, within $(2-2\rho, \frac{2}{(1-2\rho)^2}).$ 
	Combining these, we get the results.
\end{proof}
\begin{Lem}
	There exists $d_1>0$ such that,
	
	\begin{align*}
		-d_1\intb v_{\epsilon}^2(x) dx &\leq \intb g(x) v_{\epsilon}^2(x)dx \leq d_1\intb v_{\epsilon}^2(x) dx.\\
		-d_1\intb v_{\epsilon}^2(x) dx &\leq \intb (g(x)+\theta h(x)) v_{\epsilon}^2(x)dx \leq d_1\intb v_{\epsilon}^2(x) dx, 
	\end{align*}
	where $d_1<\bigg(\la+\frac{N(N-2)}{4}+\theta ln8\bigg)$ whenever $\rho\leq\frac{1}{4}$.
\end{Lem}
\begin{proof}
	Follows directly by estimating $g(x)$ and $g(x)+\theta h(x)$ when $|x|\leq \rho$.
\end{proof}
\begin{Lem} \label{critical bound}
	If $N\geq 4 $ then $ d_{2^{\star}-1}<\frac{1}{N}S^{\frac{N}{2}}$.
\end{Lem}
\begin{proof}
	The proof follows as in Deng et.al. \cite{DHPZ-2023}. We highlight the case when $N=4$. The other case can be done analogously.
	Define
	\begin{align*}
		\psi(t)=\tilde{J}(t v_{\epsilon}).
	\end{align*}
	Then $\psi'(t)=\tilde{J}'(tv_{\epsilon})(v_{\epsilon})=\frac{1}{t}\tilde{I}(t v_{\epsilon})$.
	Since $\psi(0)=0$ and $\lim_{t\to\infty}\psi(t)=-\infty$, there exists $t_{\epsilon}$ such that $\psi(t_\epsilon)=\max \psi(t)$. That is, $\tilde{I}(t_{\epsilon}v_{\epsilon})=0$. Hence
	\begin{align*}
		&t_{\epsilon}^2\intb |\nabla v_{\epsilon}|^2dx -t_{\epsilon}^2 \intb (g+\theta h)v_{\epsilon}^2dx- t_{\epsilon}^{2^{\star}}\intb v_{\epsilon}^{2^{\star}}dx\\ 
&-\theta t_{\epsilon}^2\intb h v_{\epsilon}^2\ln v_{\epsilon}^2dx-\theta t_{\epsilon}^2\ln t_{\epsilon}^2\intb v_{\epsilon}^2dx=0.
	\end{align*}
	Simplifying this, we get
	\begin{align}\label{sec3eqn1}
		\intb |\nabla v_{\epsilon}|^2dx-\intb (g+\theta h)v_{\epsilon}^2 dx-\theta\intb h v_{\epsilon}^2 \ln v_{\epsilon}^2dx=t_{\epsilon}^{2^{\star}-2}\intb |v_{\epsilon}|^{2^{\star}}dx+\theta \ln t_{\epsilon}^2\intb v_{\epsilon}^2dx. 
	\end{align}
	Using the suitable bounds in the above asymptotic estimates, we get
	\begin{align*}
		S^{\frac{N}{2}}+O(\epsilon^{N-2})+d_1 d\epsilon^2|\ln \epsilon| +C_\rho\theta \epsilon^2\ln \frac{1}{\epsilon}+O(\epsilon^2)\geq t^{2^{\star}-2}\frac{1}{2}S^{\frac{N}{2}}+\theta \ln t_{\epsilon}^2(d\epsilon^2|\ln\epsilon|+O(\epsilon^{N-2})).
	\end{align*}
	Therefore, $ t_{\epsilon} \leq C \quad as \, \epsilon\to 0^+.$
	Similarly using respective bounds from the asymptotic estimates we get
	\begin{align*}
		\frac{1}{2}S^{\frac{N}{2}}\leq t_{\epsilon}^{2^{\star}-2}\intb |v_{\epsilon}|^{2^\star}dx+\theta \ln t_{\epsilon}^2\intb v_{\epsilon}^2 dx \ .
	\end{align*}
	Hence $t_{\epsilon}$ stays away from $0,$ that is $C^{-1}<t_{\epsilon}<C$  for all $\epsilon>0$ small enough, for some constant $C>0.$ Therefore
	\begin{align*}
		d_{2^\star-1}&\leq \tilde{J}(t_{\epsilon}v_{\epsilon})\\
		&= \frac{t_{\epsilon}^2}{2}\intb |\nabla v_{\epsilon}|^2dx-\frac{t_{\epsilon}^{2^{\star}}}{2^{\star}}\intb |v_{\epsilon}|^{2^{\star}}dx-\frac{1}{2}\intb g v_{\epsilon}^2 dx-\frac{\theta}{2}\intb h v_{\epsilon}^2\ln v_{\epsilon}^2dx\\
		&\leq \left(\frac{t_{\epsilon}^2}{2}-\frac{t_{\epsilon}^{2^\star}}{2^{\star}}\right)S^{\frac{N}{2}}-\theta C_{\rho}\epsilon^2\ln\left(\frac{1}{\epsilon}\right)+d_1d\epsilon^2\ln\frac{1}{\epsilon}+O(\epsilon^2) +O(\epsilon^{N-2})\\
		&\leq \frac{1}{N}S^{\frac{N}{2}}-(\theta C_{\rho}-d_1d)\epsilon^2\ln\frac{1}{\epsilon}
		+O(\epsilon^2)+O(\epsilon^{N-2})\\
		&<\frac{1}{N}S^{\frac{N}{2}}. 
	\end{align*}
	where the last inequality follows from $C_\rho\to \infty$ as $\rho\to 0$.
\end{proof}

\section{$\theta >0$: Existence of positive Ground State Solutions}\label{existence}
In this section, we prove the existence of a positive ground state solution to \eqref{LSE} for $\theta>0$. We first consider the subcritical case $1<p<2^\star-1$, and establish the existence of a positive solution. Using this and the energy estimate proved in Section 3, we then prove the existence of a positive ground state solution in the critical case $p = 2^\star-1$.
\subsection{The sub-critical case: $1 < p <2^{\star} - 1$}
\begin{Th}\label{Existsubcritical}
	Let $1 < p <2^{\star} - 1.$ Then there exists  $u\in \mathcal{N}_p$ such that $J_p(u)=d_p.$ In particular, \eqref{LSE} admits a positive solution in $H^1(\bn)$.  
\end{Th}
\begin{proof}
	Let $\{u_n\}\subset \mathcal{N}_p$ be a minimising sequence such that $J_p(u_n)\to d_p$,
	Then,
	$J_p(u_n)=d_p +o(1)$ gives
	\begin{align}\label{Eq3.1.1}
		\frac{1}{2}\inth |\nabla_{\bn}u_n|_g^2\dx -\frac{\la}{2}\|u_n\|_2^2-\frac{1}{p+1}\|u_n\|_{p+1}^{p+1}-\frac{\theta}{2}\inth u_n^2(\ln u_n^2-1)\dx=d_p +o(1),
	\end{align}
	and $I_p(u_n)=0$ gives
	\begin{align}\label{Eq3.1.2}
		\|u_n\|_{\la}-\|u_n\|_{p+1}^{p+1}-\theta\inth u_n^2\ln u_n^2=0.
	\end{align}
	Recall the logarithmic-Sobolev Inequality; for any $u\in H^1(\bn)$ and for all $\epsilon>0$,  
	\begin{align*}
		\inth u^2 \ln u^2 \dx\leq \frac{\epsilon}{\pi}\|\nabla_{\bn} u\|_2^2+\|u\|_2^2(\ln\|u\|_2^2 +C_1-C_2\ln\epsilon).
	\end{align*}
	We first show that $\|u_n\|_{\la}\lesssim 1.$
	Multiplying  \eqref{Eq3.1.2} by $\frac{1}{2}$ and subtracting from \eqref{Eq3.1.1}, we obtain
	\begin{align*}
		\frac{\theta}{2}\|u_n\|_2^2+\left(\frac{1}{2}-\frac{1}{p+1}\right)\|u_n\|_{p+1}^{p+1}\lesssim 1,
	\end{align*}
	which yields $\|u_n\|_2^2\lesssim 1$ and $\ \|u_n\|_{p+1}^{p+1}\lesssim 1.$
	Plugging this in \eqref{Eq3.1.1} and using the logarithmic Sobolev inequality, we get
	\begin{align*}
		\frac{1}{2}\|\nabla_{\bn}u_n\|_{2}^2 &\lesssim C+\frac{\theta}{2}\inth u_n^2\ln u_n^2\dx\\
		&\lesssim C+ \epsilon\|\nabla_{\bn} u_n\|_2^2+ \|u_n\|_2^2(\ln \|u_n\|_2^2+C_1-C_2\ln\epsilon).
	\end{align*}
	Choosing $\epsilon$ small enough we deduce  $\|\nabla_{\bn} u\|_2^2\lesssim 1$. Hence up to a subsequence
	$u_n\rightharpoonup u$ in  $H^1_r(\bn), \ u_n\to u$ in  $L^q(\bn) \mbox{ for } 2<q<2^{\star}$ and a.e. in $\bn.$
	
	\medskip

	Since $I_p(u_n)=0,$ using the above bounds, we get $\inth (u_n^2\ln u_n^2)^-dx\lesssim 1$  and hence by Fatou's lemma  $\inth |u^2\ln u^2|\dx <\infty$ proving $u\in X\cup \{0\}$.
	
	\medskip
	
	Now, we prove a positive lower bound for the sequence. Let $\delta>0$ be such that $-\la u^2 + \theta (u \ln u^2)^{-} \geq 0$ for $u \leq \delta.$
	By $I_p(u_n)=0$ 
	\begin{align*}
		\|\nabla_{\bn} u_n\|_{2}^2-\|u_n\|_{p+1}^{p+1}&\leq |\la|\int_{\{u \geq \delta\}} u^2 \dx +\theta\inth (u_n^2\ln u_n^2)^+\dx \lesssim \inth u_n^{p+1}\dx.
	\end{align*}
	This, combined with the Poincar\'e-Sobolev inequality yields
	\begin{align*}
		\|u_n\|_{p+1}^2\lesssim\|u_n\|_{\la}^2\lesssim \|u_n\|_{p+1}^{p+1}.
	\end{align*} 
	Therefore $\|u_n\|_{p+1}\gtrsim C>0$. Note that
	\begin{align*} 
		J_p(u_n)-\frac{1}{2}I_p(u_n)=\frac{\theta}{2}\|u_n\|_2^2+\left(\frac{1}{2}-\frac{1}{p+1}\right)\|u_n\|_{p+1}^{p+1}.
	\end{align*}
	The right hand side of the above equation has a uniform positive lower bound  and $J_p(u_n)-\frac{1}{2}I_p(u_n)\to d_p$. Hence, we have $d_p>0$. Since $u_n$ is strongly convergent in $L^{p+1}(\bn)$, we have $\|u\|_{p+1}^{p+1}\gtrsim C$. Hence $u\not\equiv 0$ and $u\in X$.
	
	It remains to show that $I_p(u) = 0.$ By weak lower semicontinuity of $I_p$, we already have $I_p(u)\leq 0 $.
	By Lemma \ref{Uniqueness}, there exists $t\in(0,1]$, such that $I_p(tu)=0.$ We will show that $t=1$. We have
	\begin{align*}
		d_p\leq J_p(tu)=J_p(tu)-\frac{1}{2}I_p(tu)&=\bigg(\frac{1}{2}-\frac{1}{p+1}\bigg)t^{p+1}\|u\|_{p+1}^{p+1}+\frac{\theta}{2}t^2\|u\|_2^2\\
		&\leq \bigg(\frac{1}{2}-\frac{1}{p+1}\bigg)\|u\|_{p+1}^{p+1}+\frac{\theta}{2}\|u\|_2^2\\
		&\leq \liminf \ \left[\bigg(\frac{1}{2}-\frac{1}{p+1}\bigg)\|u_n\|_{p+1}^{p+1}+\frac{\theta}{2}\|u_n\|_2^2 \right]\\
		&= \lim \ \left[J_p(u_n)-\frac{1}{2}I_p(u_n) \right]=d_p .
	\end{align*}
	Hence all the inequalities in the above chain are equalities which is only possible if $t=1$ as $d_p>0$. This concludes the proof.
\end{proof}
\subsection{The critical case: $p = 2^{\star} - 1$}
Our main aim is to show that $d_{2^\star-1}$ is attained. We aim to approximate $d_{2^{\star}-1}$ by optimizers of sub-critical problems. To achieve this, we first prove a few lemmas that will help us reach our goal.
\begin{Lem}\label{limsup}
	We have 
	$$\limsup\limits_{p\to 2^{\star}-1 \atop p\in(1,2^{\star}-1)} \ d_p\leq d_{2^{\star}-1} .$$
\end{Lem}
\begin{proof}
	By definition, for every $\epsilon>0$, there exists $u\in  \mathcal{N}_{2^{\star}-1}$ such that $J_{2^\star-1}(u)<d_{2^{\star}-1}+\epsilon$.
	Let $p_n \in (1, 2^{\star}-1)$ be such that $p_n\to 2^{\star}-1$ as $n \to \infty$. Then for each $n \in \mathbb{N}$, there exists $t_n$ such that $I_{p_n}(t_n u)=0$. Expanding this we obtain
	\begin{align*}
		t_n^2\|u\|_{\la}^2= t_n^{p_n+1}\|u\|_{p_n+1}^{p_n+1}+\theta  t_n^2 \inth u^2\ln(t_n^2 u^2)\dx,
	\end{align*}
	which immediately gives that $|t_n|\lesssim 1$.
	Thus up to a subsequence we have, $t_n\to t_0 \in \mathbb{R}$. Then it is easy to prove that
	\begin{align*}
		I_{p_n}(t_nu)\to I_{2^{\star}-1}(t_0u),
	\end{align*}
	which yields $I_{2^{\star}-1}(t_0 u)=0$. Hence by Lemma \ref{Uniqueness}, we have $t_0=1$. This is true for every subsequence of $t_n $. Since every subsequence of the sequence $\{t_n\}$ has a further subsequence converging to the unique limit 1, the whole sequence also converges to 1.
	\begin{align*}
		\limsup_{n\to\infty} d_{p_n}\leq \lim_{n \to \infty} J_{p_n}(t_nu)=J_{2^{\star}-1}(u)<d_{2^{\star}-1}+\epsilon
	\end{align*}
	Since $\epsilon$ is arbitrary, this  concludes the proof. 
\end{proof}

\begin{Lem} \label{cpt 2star}
	Let $1 < p \leq 2^{\star} - 1$ and let $u$ be a  solution to \eqref{LSE}. Assume that there exists a $\delta >0,$ such that 
	$$\|u\|_{p+1}^{p+1}<(1-\delta)^{\frac{N}{2}}S^{\frac{N}{2}} \ \ 
	\mbox{and} \ \ \|\nabla_{\bn} u\|_2^2\lesssim 1 \ .$$
	Then, there exists an $r = r(\delta)>2^{\star}$ such that $u\in L^r(\bn)$ and $\|u\|_r \lesssim_ {\delta} 1.$ 
\end{Lem}
\begin{proof}
	We follow Br\'ezis-Kato's argument.
	Define, for each $L>1$, $\phi= u \min\{|u|^{2s},L^2\},$ where $s >0$ will be fixed later. Then $|\nabla_{\bn} \phi|_g=|\nabla_{\bn} u|_g(2s|u|^{2s-1}\chi_{\{|u|^{2s}<L^2\}}+\min\{|u|^{2s},L^2\})$ $\in L^2(\bn)$ and hence $\phi \in H^1(\bn)$. Now, we show that $\phi\in X$ for $4s+2<2^{\star}.$ We have
	\begin{align*}
		\inth|\phi^2\ln \phi^2|\dx&=\int_{|u|<1} |\phi ^2 \ln \phi^2|\dx \\
		&\ + \int_{1<|u|^{2s}<L^2}|\phi^2 \ln\phi^2|\dx+\int_{|u|^{2s}\geq L^2} |\phi^2 \ln \phi^2|\dx.
	\end{align*}
	Since $u\in X$,  $L^2 u$ is also in $X$, and therefore the last term is finite. The other two terms are finite due to the following estimates:
	\begin{align*}
		\begin{cases}
			|\phi^2 \ln\phi^2|\lesssim u^{(1-\epsilon)(4s+2)} \quad \text{for}\quad |u|<1,\\
			|\phi^2 \ln\phi^2|\lesssim u^{(1+\epsilon)(4s+2)} \quad \text{for}\quad |u|\geq 1,
		\end{cases}
	\end{align*}
	where $\epsilon$ is chosen such that $(1-\epsilon)(4s+2)>2$ and $(1+\epsilon)(4s+2)<2^{\star}$.
	Hence, $\phi$ is a suitable test function for the weak formulation 
	\begin{align}\label{test fn}
		\inth\langle\nabla_{\bn} u, \nabla_{\bn} \phi\rangle_g\dx=\inth |u|^{p-1}u\phi \dx +\theta  \inth u\phi \ln u^2\dx +\la \inth u\phi\dx.
	\end{align}
	Next, we estimate each term. A straight forward computation gives
	\begin{align*}
		\langle\nabla u,\nabla \phi\rangle_g=\langle \nabla u,\nabla u\rangle_g \min\{|u|^{2s},L^2\}+2s\langle\nabla u,\nabla u\rangle_g |u|^{2s-1}\chi_{\{|u|^{2s}<L^2\}},
	\end{align*}
	\begin{align*}
		\langle\nabla(u\phi)^{\frac{1}{2}},\nabla(u\phi)^{\frac{1}{2}}\rangle_g &=\langle\nabla u,\nabla u\rangle_g (\min\{|u|^s,L\})^2+ s^2\langle\nabla u, \nabla u\rangle_g |u|^{2s}\chi_{\{|u|^{2s}<L^2\}} \\
		& \ \ \ \ +2s\langle\nabla u, \nabla u\rangle_g |u|^{2s-1}\chi_{\{|u|^{2s}<L^2\}}.
	\end{align*}
	Using these, the left hand side of \eqref{test fn} can be estimated as
	\begin{align*}
		\inth\langle\nabla_{\bn} u,\nabla_{\bn} \phi\rangle_g\dx &=\inth \langle \nabla_{\bn} (u\phi)^{\frac{1}{2}},\nabla_{\bn}(u\phi)^\frac{1}{2}\rangle_g\dx\\
		&\quad-s^2\inth\langle \nabla_{\bn} u,\nabla_{\bn} u\rangle_g u^{2s}\chi_{\{|u|^{2s}<L^2\}}\dx\\
		&\geq (1-s^2)\inth\langle \nabla_{\bn}(u\phi)^{\frac{1}{2}},\nabla_{\bn}(u\phi)^\frac{1}{2}\rangle_g\dx.
	\end{align*}
	Whereas the right hand side of \eqref{test fn} can be estimated as
	\begin{align*}
		&\inth |u|^{p-1}u \phi\dx +\theta \inth\phi (u\ln u^2)\dx +\la \inth u\phi\dx \\
		= &\int_{|u|>1} |u|^{p-1}u \phi\dx+ \int_{u<1}|u|^{p-1}u\phi \dx +\la\inth u\phi\dx +\theta \inth \phi (u\ln u^2)\dx \\
		\leq & \left(\int_{|u|>1}|u|^{(p-1)\frac{N}{2}}\dx\right)^{\frac{2}{N}}\left(\int_{|u|>1}(u\phi)^\frac{N}{N-2}\dx\right)^\frac{N-2}{N} +\int_{u<1} u\phi \dx +\la \inth u\phi\dx\\
		&\quad+\theta \inth \phi(u\ln u^2)\dx \\		
		\leq & \left(\int_{|u|>1} |u|^{p+1}\dx\right)^{\frac{2}{N}\dx}\left(\int_{|u|>1} (u\phi)^{\frac{N}{N-2}}\dx\right)^{\frac{N-2}{N}}+(\la+1)\inth u\phi\dx\\
		&+ \theta\inth \phi (u \ln u^2)\dx\\
		< & (1-\delta) S \left( \inth ( u\phi)^{\frac{N}{N-2}}\dx \right)^\frac{N-2}{N} +(\la+1)\inth u\phi \dx +\theta \int_{|u|>1}(\phi u\ln u^2)^+\dx,
	\end{align*}
	here, in the second last inequality, we used the condition $(p-1)\frac{N}{2}\leq p+1.$  
	Combining the last two inequalities, we get
	\begin{align*}
		&(1-s^2)\inth \langle\nabla_{\bn}(u\phi)^{\frac{1}{2}},\nabla_{\bn}(u\phi)^{\frac{1}{2}} \rangle_g\dx\\
		&\quad \leq (1-\delta)S\bigg(\inth(u\phi)^\frac{N}{N-2}\dx\bigg)^\frac{N-2}{N}+(\la+1)\inth u\phi\dx +\theta \int_{|u|>1}(\phi u\ln u^2)^+\dx .
	\end{align*}
	Choose $s>0$ such that $\delta-s^2>\frac{\delta}{2}.$ Then, by Poincar\'e-Sobolev inequality,
	\begin{align*}
		\bigg(\inth(u\phi)^{\frac{N}{N-2}}\dx\bigg)^{\frac{N-2}{N}} &\lesssim_{\delta} |\la+1|\inth\bigg( u^2+ |u|^{2^\star}\dx\bigg)+\theta\inth |u|^{2^{\star}}\dx\\
		&\lesssim_{\delta} 1.
	\end{align*}
	Letting $L\to \infty$, we conclude
	\begin{align*}
		\bigg(\inth |u|^{(s+1)2^{\star}}\dx \bigg)^{\frac{N-2}{N}} \lesssim_{\delta} 1.
	\end{align*}  
	This completes the proof with $r = (s+1)2^{\star}.$
\end{proof}

\begin{Rem}\label{rembreziskato}
	The same Br\'ezis-Kato argument as described above shows that, if $u \in H^1(\bn)$ is a solution to \eqref{LSE}, then $u \in L^{\infty}(\bn),$ irrespective of the values of $\theta.$ Indeed, the case $\theta >0$ has been described above. For $\theta <0,$ we cannot drop the  $-\theta \int_{\{u \leq 1\}} (\phi u \ln u^2)^-$ term. However, since we assume that $u \in H^{1}(\bn),$ approximating $u$ by $C_c^{\infty}$ functions and passing to the limit in the weak formulation we can conclude that 
	\begin{align*}
	-\theta \inth ( u^2 \ln u^2)^- \dx &= \inth (|\nabla_{\bn} u|_g^2 - \la u^2 - |u|^{p+1} - \theta (u^{2}\ln u^2)^{+})\dx \\
	&\lesssim 1.	
	\end{align*}
	Hence, $u^2\ln u^2 \in L^1(\bn).$ As a result, we can estimate the term $-\theta \int_{\{u \leq 1\}} (\phi u \ln u^2)^-\dx $ uniformly by $\|u^2\ln u^2\|_{1}$ and $\|u\|_{2^{\star}}.$ Hence, by translation invariance of the problem and elliptic regularity, we conclude that if $u \in H^1(\bn)$ is a solution to \eqref{LSE}, $\theta \in \mathbb{R}$ then $u(x) \rightarrow 0$ as $|x| \rightarrow 1.$
\end{Rem}

\begin{Th}
	There exists $u\in \mathcal{N}_{2^{\star}-1}$ such that $J_{2^\star-1}(u)=d_{2^\star-1}$. In particular, the equation \eqref{LSE} admits a solution for $ p = 2^{\star} - 1.$
\end{Th}
\begin{proof}
	Choose a sequence $\{p_n\} \subset (1, 2^{\star}-1)$ such that $p_n \rightarrow 2^{\star}-1.$
	By Theorem \ref{Existsubcritical}, there exists $u_n\in  \mathcal{N}_{p_n}$, $u_n>0$ such that, $J_{p_n}(u_n)=d_{p_n}$.
	Now using Lemma \ref{limsup}, we have
	\begin{align*}
		\left(\frac{1}{2} - \frac{1}{p_n+1}\right)\|u_n\|_{p_n+1}^{p_n + 1}  \leq |J_{p_n}(u_n)-\frac{1}{2}I_{p_n}(u_n)|=d_{p_n} \leq d_{2^{\star}-1}+ o(1),
	\end{align*}
	as $n \rightarrow \infty.$ By Lemma \ref{critical bound}, there exists a $\delta >0,$ and $n_0,$ such that $\|u_n\|_{p_n+1}^{p_n + 1}  \leq (1 - \delta)^{\frac{N}{2}}S^{\frac{N}{2}}$ for all $n \geq n_0.$
	Following the same proof using the Log-Sobolev inequality of the subcritical case we conclude $\|\nabla_{\bn}u_n\|_2 \lesssim 1.$ By Lemma \ref{cpt 2star}, $\|u_n\|_r$ is uniformly bounded for some $r > 2^{\star}.$
	Hence, up to a subsequence, we have $u_n\rightharpoonup u$ in $H^1_r(\bn)$, $u_n\to u$ in $L^q(\bn)$ for $q\in(2,2^{\star}]$, and a.e. in $\bn.$
	
	\medskip
	
	As before $I_{p_n}(u_n)=0$ and using the above bounds, we get $\inth (u_n^2\ln u_n^2)^-\lesssim 1$  and hence by Fatou's lemma, we have, $\inth |u^2\ln u^2|\dx <\infty$ proving $u\in X\cup \{0\}$. Similarly as before we have $\|u_n\|_{p_n + 1} \gtrsim 1$ and by strong $L^{2^{\star}}$ convergence we conclude $\|u\|_{2^{\star}} \gtrsim 1.$ Hence $u\not \equiv 0$ and $u\in X.$ Moreover, by lower semicontinuity, $I_{2^{\star} - 1}(u) \leq 0.$

	\medskip
	
	By Lemma \ref{Uniqueness}, there exists $t\in(0,1]$ such that $I_{2^\star - 1}(tu)=0$, and hence 
	$d_{2^{\star}-1}\leq J_{2^{\star}-1}(tu).$ Now, note that 
	\begin{align*}
		J_{2^{\star}-1}(tu)=J_{2^{\star}-1}(tu)-\frac{1}{2}I_{2^{\star}-1}(tu) 
		&= \frac{\theta}{2} t^2\|u\|_2^2+\bigg(\frac{1}{2}-\frac{1}{2^{\star}}\bigg)t^{2^{\star}}\|u\|^{2^{\star}}_{2^{\star}}\\
		&\leq \frac{\theta}{2} \|u\|_2^2+\bigg(\frac{1}{2}-\frac{1}{2^{\star}}\bigg)\|u\|^{2^\star}_{2^{\star}}\\
		&\leq \liminf_{n \to \infty} \left[\frac{\theta}{2} \|u_n\|_2^2+\bigg(\frac{1}{2}-\frac{1}{p_n+1}\bigg)\|u_n\|^{p_n+1}_{p_n+1}\right]\\
		&=\liminf_{n \to \infty} \left(J_{p_n}(u_n)-\frac{1}{2}I_{p_n}(u_n)\right) = \liminf_{n\to\infty} d_{p_n}.
	\end{align*}
	Combining the last two inequalities together we get,
	\begin{align*}
		d_{2^\star-1}&\leq J_{2^{\star}-1}(tu)
		\leq \liminf_{n \to \infty} d_{p_n} \leq \limsup_{n \to \infty} d_{p_n} \leq d_{2^\star-1}.
	\end{align*}
	Hence the above chain of inequalities are equalities and $t=1$. %as $d_{2^{\star}-1}> 0$. 
	This concludes the proof.
\end{proof}

We can even say more: for $\theta>0$, a positive radial energy solution to \eqref{LSE} is actually strictly decreasing and decay to zero at infinity. Note that in the radial case, \eqref{LSE} can be written as 
\begin{align}\label{Radial} 
	u''(\rho)+ \frac{(N-1)}{\tanh \rho} u'(\rho) + \la u(\rho) + u^p(\rho) + \theta u(\rho) \ln u^2(\rho) =0 ,  \ \ u'(0) =0,
\end{align}
where $\rho= d(x, 0) = \log \left(\frac{1 + |x|}{1 - |x|}\right)$, $\la \in \mathbb{R}$. For this subsection, by abuse of notation we will write $u(x) = u(\rho),$ whenever $u$ is radial.
\begin{Lem}
	Let  $u \in H^1(\bn)$ be a radial solution to \eqref{LSE} and $\theta>0$ then 
		$u^{\prime}(\rho)<0$ for every $\rho>0$ and $\lim_{\rho \to \infty} u(\rho) = \lim_{\rho \to \infty } u^{\prime}(\rho) =0$.
\end{Lem}
\begin{proof}
	Inspired from \cite{MS-2008}, we define the energy functional corresponding to \eqref{Radial} by
	\begin{equation*}
		E_u(\rho) = \frac{u'^2}{2}(\rho) + \frac{\la}{2} u^2(\rho) + \frac{|u|^{p+1}}{p+1}(\rho) + \frac{\theta}{2} u^2(\rho) (\ln u^2(\rho) -1).
	\end{equation*}
A direct computation gives $\frac{d}{d\rho}E_u(\rho) = -\frac{(N-1)}{\tanh \rho} u^{\prime 2}(\rho) \leq 0$, for all $\rho >0$. Since $u$ is an energy solution 
\begin{align*}
		\|u\|_{H^1}^2 + \|u\|_2^{2}= \omega_{N-1} \int_0^\infty (u^{\prime 2}(\rho) + u^2(\rho))\sinh^{N-1} \rho  \ d\rho  < \infty,
	\end{align*}
and hence
	\begin{equation}\label{eqn4}
		\liminf_{\rho \to \infty} \   [u^{\prime 2}(\rho) + u^2(\rho)] \sinh^{N-1} \rho=0.
	\end{equation}
	Then by \eqref{eqn4} and the monotonicity of $E_u$, we conclude  $E_u(\rho)\geq0,$ for all $\rho>0$.
	Now we claim $E_u(\rho) > 0$ for all $\rho\geq0.$ If for some $\rho_1, E_u(\rho_1) = 0$ then $E_u(\rho)\ = 0$ for all $\rho \geq \rho_1$ and so does its derivative. Hence $u^{\prime}(\rho) = 0$ for all  $\rho \geq \rho_1$ and by \eqref{eqn4}, we get $u = 0$ for all $\rho \geq \rho_1,$ a contradiction. First assume that $u^{\prime}(\rho_0)=0$ for some $\rho_0> 0$. Then 
	\begin{align*}
		0 < E_u(\rho_0) = \frac{\la}{2} u^2(\rho_0) + \frac{|u|^{p+1}(\rho_0)}{p+1} + \frac{\theta}{2} u^2(\rho_0) (\ln u^2(\rho_0) -1),
	\end{align*} 
	and since $p>1$ we get
	\begin{align*}
		\la u(\rho_0) + u^{p+1}(\rho_0)+ \theta u^2(\rho_0)(\ln u^2(\rho_0) -1) >0.
	\end{align*}
	By equation \eqref{Radial} we get, $u^{\prime\prime}(\rho_0)<0$.Therefore $u^{\prime}(\rho)$ must be $> 0$ in a small neighbourhood $(\rho_0-\epsilon, \rho_0).$ Hence  $u^{\prime}(\rho)>0$ in $(0,\rho_0)$. Since $u^{\prime\prime}(0)<0$, we have  $u^{\prime}(\rho) < 0$ in a neighbourhood of $(0,0+\epsilon)$	
which is absurd. Therefore $u^{\prime}(\rho)<0$ and using \eqref{eqn4}, we get the asymptotic decay of $u$ and $u^{\prime},$ completing the proof. 
\end{proof}

\section{$\theta <0:$ Nonexistence results}\label{nonexistence}
In this section, we prove that under the assumption $\theta<0,$ there is no positive energy solution to \eqref{LSE}, irrespective of the values of $\la.$ Recall that by an energy solution we mean that $u \in H^1(\bn).$
The main result of this section is a lower asymptotic decay estimate on the positive energy solutions.
Note that we do not assume  $u^2\ln u^2 \in L^1(\bn).$
Indeed, if we have assumed $u^2\ln u^2 \in L^1(\bn),$ then it is expected that a positive energy solution must be radial (with respect to some point say $0$). In particular, $u$ would have the radial decay 
$$
u(x) \lesssim (1 - |x|^2)^{\frac{N-1}{2}}, \ \ x \in \bn.
$$ 
 In Lemma \ref{A lower bound}, we obtain the opposite inequality on any positive solution. Hence for radial energy solutions we have the precise decay $u(x) \approx (1 - |x|^2)^{\frac{N-1}{2}}.$
The next few basic lemmas need for the proof of Theorem \ref{main}(b) and \ref{main2}.

\subsection{A Subsolution}
\begin{Lem} \label{sub}
Let $\la_0>\frac{(N-1)^2}{4}$, then there exists a constant $R_{\la_0}$ depending on $\la_0$ and $N$ such that for every $\la \geq \la_0,$ the function $u(x)=\bigg(\sinh\frac{d(0,x)}{2}\bigg)^{-(N-1)}$ satisfies
\begin{align*}
-\De_{\bn} u-\la u\leq 0
\end{align*}
 in $\{x \in \bn \ | \ d(0,x)>R_{\la_0}\}.$
\end{Lem}
\begin{proof}
 We denote the radial coordinate by $\rho(x)=d(0,x).$ For simplicity we shall denote $u(x) = u(\rho).$ A straightforward computation gives (details can be found in Appendix)
\begin{align*}
-\De_{\bn} u-\la u \ = \ &\frac{(N-1)}{4}\left((N-2)\coth^2\frac{\rho}{2}+1-\frac{4}{(N-1)}\la\right)\left(\sinh\frac{\rho}{2}\right)^{-(N-1)} \\
& \ \ \ \ -o\left(\sinh\frac{\rho}{2}\right)^{-(N-1)}
\end{align*}
as $\rho \rightarrow \infty.$ Since $\coth \frac{\rho}{2} \rightarrow 1$ as $\rho \rightarrow \infty$
and $\la \geq \la_0 > \frac{(N-1)^2}{4},$ we conclude the proof.
\end{proof}
Note that $u$ in Lemma \ref{sub} is not in $H^1(\bn).$ However, for $\epsilon >0$
$$
u_{\epsilon}(\rho)=\left(\sinh\frac{\rho}{2}\right)^{-(N-1+\epsilon)}
$$ 
are $H^1$-functions. We also set
 $$f(\rho,\epsilon)=-u_{\epsilon}^{\prime\prime}(\rho)- (N-1)\coth\rho \ u_{\epsilon}^{\prime}(\rho)-\la u_\epsilon(\rho).$$

\begin{Lem}\label{epsilon-pertubed-subsoln}
Let $\la_0 > \frac{(N-1)^2}{4},$ and let $R_{\la_0}$ be as in Lemma \ref{sub}. Then  
there exists $\epsilon_0>0$ such that $f(\rho,\epsilon)<0$ for all $\la \geq \la_0,$ $\rho \geq R_{\la_0}$ and $\epsilon<\epsilon_0$.
\end{Lem}
\begin{proof}
A detailed computation which can be found in the Appendix confirms that 
\begin{align*}
f(\rho,\epsilon)= \left(\sinh\frac{\rho}{2}\right)^{-\epsilon}f(\rho,0)+\epsilon \  O\left(\left(\sinh\frac{\rho}{2}\right)^{-(N-1+\epsilon)}\right)
\end{align*}
as $\rho \rightarrow \infty.$ By Lemma \ref{sub}, $f(\rho, 0)$ behaves like $\frac{(N-1)}{4}\left((N-1)-\frac{4}{(N-1)}\la\right)\left(\sinh\frac{\rho}{2}\right)^{-(N-1)}$ as $\rho \rightarrow \infty,$ we conclude the proof.
\end{proof}
\subsection{Asymptotic Estimate}
\begin{Lem}{(Picone's inequality)}
Let  $u,v\in W^{1,2}_{loc}(\bn)$ then
\begin{align*}
B(u,v)&:=\langle\nabla_{\bn} u,\nabla_{\bn}(u-\frac{v^2}{u^2}u)\rangle_g+\langle\nabla_{\bn} v,\nabla_{\bn}(v-\frac{u^2}{v^2}v)\rangle_g\\
&\geq
\min\{u^2,v^2\}|\nabla_{\bn}(\ln u-\ln v)|_g^2.
\end{align*}
\end{Lem}
The inequality  is a direct consequence of multiplying the conformal factor to the Euclidean  identity
\begin{align*}
\frac{|\nabla v|^2}{v^2}-|\nabla(\ln v-\ln u)|^2=\frac{\nabla u}{v^2}\nabla\frac{v^2}{u}.
\end{align*}
The one-dimensional version was used by M. Picone in \cite[Section 2]{P-1910}  to prove the Sturm-comparison theorem. The identity for general exponent can be found in \cite[Lemma 3.1]{X-2015} and \cite[Lemma 3.1]{OSV-2020}  for the Euclidean case and in \cite[Lemma 3.4]{DS-2024} for the hyperbolic case. See also \cite{BT-2021} for a generalized Picone's identity and it's applications.

\medskip

Now we can state and prove a precise lower bound of the $H^1$-solution, which will lead us to the proof of our main non-existence theorems.

\begin{Lem}\label{A lower bound}
Let $u \in H^1(\bn)$ be a positive solution to \eqref{LSE} with $\theta <0$ and $\la \in \mathbb{R}.$
There exists an $R_0>0$ and $C_0>0$ such that 
\begin{align*}
C_0 \ sinh \left(\frac{d(0,x)}{2}\right)^{-(N-1)}\leq u(x)\quad ,\, \forall x\in \bn\setminus B_{R_0}.
\end{align*}
\end{Lem}

\begin{proof}
We work in geodesic normal coordinate and let $\rho(x)=d(0,x)$. Fix $\la_0 > \frac{(N-1)^2}{4}.$ There exists $\gamma>0$ such that  
\begin{align*}
\la + \theta \ln u^2 >\la_0,\, \mbox{whenever} \ \ u\leq\gamma. 
\end{align*}
By Lemma \ref{epsilon-pertubed-subsoln}, there exists $\epsilon_0>0, R_{\la_0}$ such that  for $\epsilon<\epsilon_0$ and $\rho\geq R_{\la_0}$,  $v(x) := (\sinh\frac{\rho}{2})^{-(N-1+\epsilon)}$ satisfies,
\begin{align*}
-\De_{\bn}v -\la_0 v <0.
\end{align*}
Now, let $R_1=2\sinh^{-1}(\gamma^{-\frac{1}{N-1}})$ and $R_0=\max\{R_{\la_0},R_1\}$. We appropriately define $v$ on $\bn$ by
$v_\epsilon(\rho)=\min
\{(\sinh(\frac{\rho}{2})^{-(N-1+\epsilon)},(\sinh(\frac{R_0}{2})^{-(N-1+\epsilon)}\}$ for all $\epsilon\in [0,\epsilon_0]$. Observe that $v_{\epsilon}\in H^{1}(\bn),v_{\epsilon}\leq \gamma$ for all $0<\epsilon\leq \epsilon_0$ and $v_{\epsilon}$ is a decreasing family in $\epsilon$.
As $u$ is smooth and strictly positive there exists a $C_0\in(0,1)$ such that, $C_0 v_{\epsilon} \leq u$ on $\overline{B_{R_0}(0)}$. Note that the constant $C_0$ depends on $u, R_0,\la_0,N$ but can be chosen independent of $\epsilon.$ 
Let $w_{\epsilon}=C_0 v_{\epsilon}$. Then $w_\epsilon$ is a sub solution of the equation
$-\De_{\bn} w- \la_0 w=0$ i.e., $w_{\epsilon}$ satisfies,
\begin{align}\label{subsolution}
-\De_{\bn}w_{\epsilon} - \la_0 w_{\epsilon}\leq 0 \ \ \mbox{on} \ \rho \geq R_0.
\end{align}
Since  $u>0$ solves $-\De_{\bn}u-(\la+\theta \ln u^2)u=|u|^{(p-1)}u$, $u$ is a supersolution of the equation
$-\De w - \la_0 w=0$ i.e., $u$ satisfies,
\begin{align}\label{supersolution}
-\De_{\bn}u-\la_0 u>0 \quad  \ \ \mbox{ whenever } \ u \leq \gamma.
\end{align} 
and in particular on $\{w_{\epsilon}\geq u\}.$ Note that according to our choice of $C_0$ the set $ \{w_{\epsilon}\geq u\}$ is contained in $\{\rho \geq R_0\}.$
Now set $R > R_0,$ and we choose a cutoff $\eta \equiv 1 $ in $B_R$ and
$\eta\equiv 0$ in $\bn \setminus B_{R+1} $ and $0\leq\eta\leq 1.$

\begin{comment}
Now assume that the solution $u$ is non-negative radial and {\color{red}$u> 0$ on the annulus $\{R_\la\leq\rho\leq\overline{R}\}$(we need strict positivity for choosing the $C_{\la}$)} and choose a $C_\la>0$ such that $v_{\la+\epsilon}=C_{\la}v_{\epsilon}\leq u$ on $\{R_{\la}\leq\rho\leq \overline{R}\}$
\end{comment}
 
Testing the inequality \eqref{subsolution} against $\phi_1=\eta w_{\epsilon}^{-1}(w_{\epsilon}^2-u^2)^+$, 
 the inequality \eqref{supersolution} against 
$\phi_2=\eta u^{-1}(w_{\epsilon}^2-u^2)^+$
and subtracting them we get

\begin{align*}
\int_{\bn\cap\{w_{\epsilon}\geq u\}}\eta B(w_{\epsilon},u)\dx &\leq \int_{B_{R+1}\setminus B_R}|\nabla_{\bn}\eta|_g(|\nabla_{\bn} u|_g u^{-1}(w_{\epsilon}^2-u^2)^+)\dx \\
&\ \ \ \ \ \ \ \ + \int_{B_{R+1}\setminus B_R} 
|\nabla_{\bn}\eta|_g(|\nabla _{\bn}\ln w_{\epsilon}|_g (w_{\epsilon}^2-u^2)^+)\dx\\
&= \Romannum{1}_R+\Romannum{2}_R \ .
\end{align*}
By Picone's inequality we have
\begin{align*}
\int_{\{w_{\epsilon}\geq u\}}\eta B(w_{\epsilon},u)\dx \geq \int_{ \{w_{\epsilon}\geq u\}}\eta u^2|\nabla_g(\ln w_{\epsilon}-\ln u)|_g^2\dx \ .
\end{align*}
Passing through the limit $R\to \infty$ and using monotone convergence theorem we have 
\begin{align*}
\int_{ \{w_{\epsilon}\geq u\}} u^2|\nabla_g(\ln w_{\epsilon}-\ln u)|_g^2\dx \leq \limsup_{R\to \infty } \ (\Romannum{1}_R +\Romannum{2}_R).
\end{align*} 
 Now we claim to show that $\displaystyle\limsup_{R \to \infty } \ (\Romannum{1}_R +\Romannum{2}_R) = 0.$ Then we can conclude that either $u=0$ or $\ln w_{\epsilon}-\ln u=c$ on $\{w_{\epsilon}\geq u\}$. 
 
 As $u>0$ and $u$ is continuous, we have $u=w_{\epsilon}$ on the set $\{w_{\epsilon} \geq u\}$. Since the constants $C_0,R_0$ do not depend on $\epsilon,$ letting $\epsilon\to 0$ in the point wise estimate  we get the desired lower bound.

 Therefore to conclude the proof it is enough to show $\Romannum{1}_R \to 0$ and $\Romannum{2}_R \to 0$ as $R\to\infty$. The later vanishes at infinity is an easy consequence of the facts $|\nabla_g\ln w_{\epsilon}|_g<C,(w_{\epsilon}^2-u^2)^{+}\leq w_{\epsilon}^2$ and $w_{\epsilon}^2\in L^1(\bn)$.
The vanishing of $\Romannum{1}_R$ at infinity can be realised with the $L^1$ bound of the term $|\nabla u|_g^2|u^{-2}|(w_{\epsilon}^2-u^2)^{+},$ whose proof is reminiscent of the Cacciopolli inequality. We claim
\begin{align*}
\inth\frac{|\nabla_{\bn} u|^2}{u^2}(w_{\epsilon}^2-u^2)^+\dx \leq C\bigg(%\inth(w_{\epsilon}^2-u^2)^++
\|u\|_{H^1}^2+\|w_{\epsilon}\|_{H^1}^2\bigg),
\end{align*}  
where $C$ is a dimensional constant. We note that
\begin{align}\label{superharmonic}
-\De_{\bn} u\geq 0 \quad on \ \{w_{\epsilon}\geq u\} \subset \{ u<\gamma\}.
\end{align}
For $r>0,$ define a cutoff function $\phi\equiv 1$ in $B_r$, $\phi \equiv 0$ in $\bn\setminus B_{r+1}$, $0\leq \phi\leq 1$ and $|\nabla \phi|_g\approx 1$. Fix $\delta >0.$ 
We test the inequality \eqref{superharmonic} against $\frac{\phi^2}{(u+\delta)}(w_{\epsilon}^2-u^2)^+\in H^1(\bn)$ to obtain
\begin{align*}
\inth \left\langle\nabla_{\bn} u,\nabla_{\bn}\left(\frac{\phi^2(w_{\epsilon}^2-u^2)^+}{u+\delta}\right)\right\rangle_g\dx &\geq 0.
\end{align*}
Expanding the terms we get
\begin{align*}
&\ \ \ \ 2\inth \langle \nabla_{\bn} u,\nabla_{\bn} \phi\rangle_g\frac{\phi}{u+\delta}(w_{\epsilon}^2-u^2)^+\dx+2\int_{\{w_{\epsilon}^2>u^2\}} \langle\nabla_{\bn} u,\nabla_{\bn} w_{\epsilon}\rangle_g w_{\epsilon} \frac{\phi^2}{u+\delta}\dx \\
&\geq \ \  \inth \langle \nabla u ,\nabla u\rangle_g\frac{\phi^2}{(u+\delta)^2}(w_{\epsilon}^2-u^2)^+\dx
 +2\int_{\{w_{\epsilon}^2\geq u^2\}} \langle \nabla u,\nabla u\rangle_g u \frac {\phi^2}{u+\delta} \dx .
\end{align*}
Neglecting the non negative term,  $2\displaystyle\int_{\{w_{\epsilon}^2\geq u^2\}} \langle \nabla_{\bn} u,\nabla_{\bn} u\rangle_g u \frac {\phi^2}{u+\delta}\dx$ and dividing by 2, we get
\begin{align*}
\frac{1}{2}\inth \langle \nabla_{\bn} u ,\nabla_{\bn} u\rangle_g\frac{\phi^2}{(u+\delta)^2}(w_{\epsilon}^2-u^2)^+\dx &\leq \inth \langle \nabla_{\bn} u,\nabla_{\bn} \phi\rangle_g\frac{\phi}{u+\delta}(w_{\epsilon}^2-u^2)^+\dx\\
& \ \ +\int_{\{w_{\epsilon}^2\geq u^2\}} \langle\nabla u,\nabla w_{\epsilon}\rangle_g w_{\epsilon} \frac{\phi^2}{u+\delta}\dx\\
&= \Romannum{1}+\Romannum{2} \ .
\end{align*}

Now by Cauchy-Schwartz
\begin{align*}
\Romannum{1}&\leq \frac{1}{8}\inth|\nabla u|_g^2\frac{\phi^2}{(u+\delta)^2}(w_{\epsilon}^2-u^2)^+\dx + 4\inth |\nabla \phi|^2 (w_{\epsilon}^2-u^2)^+ \dx\ ,
\end{align*}
 and
\begin{align*}
\Romannum{2}&=\int_{\{w_{\epsilon}^2\geq u^2\}} \langle\nabla u,\nabla w_{\epsilon}\rangle)w_{\epsilon} \frac{\phi^2}{(u+\delta)}\dx\\
&\leq \frac{1}{8}\int_{\{w_{\epsilon}^2\geq u^2\}} |\nabla_{\bn} u|_g^2\frac{\phi^2}{(u+\delta)^2}w_{\epsilon}^2\dx +4\inth |\nabla_{\bn} w_{\epsilon}|_g^2 \phi^2\dx.
\end{align*}

 \begin{comment}
\begin{align*}
\Romannum{2}&=\int_{\{w_{\epsilon}^2\geq u^2\}} (\langle\nabla u,\nabla w_{\epsilon}\rangle-\la_1 u w_{\epsilon})w_{\epsilon} \frac{\phi^2}{(u+\delta)}+\la_1\int_{\{w_{\epsilon}^2\geq u^2\}} u w_{\epsilon}^2 \frac{\phi^2}{(u+\delta)}\\
&\leq \frac{1}{8}\int_{\{w_{\epsilon}^2\geq u^2\}} |\nabla u|_g^2\frac{\phi^2}{(u+\delta)^2}(w_{\epsilon}^2-u^2)^++\frac{1}{8}\inth (|\nabla u|_g^2-\la_1 u^2) \frac{\phi^2}{(u+\delta)^2}u^2\\
&\quad+4\inth |\nabla w_{\epsilon}|_g^2 \phi^2+\la_1\inth w_{\epsilon}^2\phi^2
\end{align*}
\end{comment}
Combining the last three inequalities and using $w_{\epsilon}^2 = (w_{\epsilon}^2-u^2)^+ + u^2$ on the domain of the integrand we get
\begin{align*}
\frac{1}{4}&\inth \langle \nabla u ,\nabla u
\rangle_g\frac{\phi^2}{(u+\delta)^2}(w_{\epsilon}^2-u^2)^+\dx\\ 
\leq 4&\inth |\nabla \phi|_g^2 (w_{\epsilon}^2-u^2)^+\dx  +\frac{1}{8}\inth (|\nabla u|_g^2) \frac{\phi^2}{(u+\delta)^2}u^2\dx+4\inth |\nabla w_{\epsilon}|_g^2 \phi^2\dx  \\
\leq 4&\bigg(\inth (w_{\epsilon}^2-u^2)^+\dx +\inth (|\nabla u|_g^2) \dx +\inth |\nabla w_{\epsilon}|_g^2\dx+\inth w_{\epsilon}^2\dx\bigg)\\
\leq C&\bigg(\|u\|_{H^1}^2+\|w_{\epsilon}\|_{H^1}^2\bigg),
\end{align*}
where the last inequality follows from $(w_{\epsilon}^2-u^2)^+\leq w_{\epsilon}^2$. Now letting $\delta \to 0$, $r\to \infty$
and using monotone convergence theorem we get the required estimate.    
\end{proof}
\begin{Rem}
The above proof can be simplified by existence of $R_1>0$ such that $\la+\theta \ln u^2(\rho)\geq \la_0 $  for all $\rho>\la_0$. This can be assumed whenever $u\to 0$ as $\rho \to \infty,$ which is true in our case  by Remark \ref{rembreziskato}. However the lemma can be proved without assuming such decay of the solution and hence can be applied in more general context.    
\end{Rem}

\medskip

\noindent
{\bf Proof of Theorem \ref{main}(b).}
\begin{proof}
Thanks to Lemma \ref{A lower bound} we conclude $u\notin L^2(\bn)$ and hence not in $H^1(\bn).$ This completes the proof. 
\end{proof}

\medskip

Now we show that any positive solution $u$ of $\eqref{LSE}$ cannot be in $\mathcal{H}^1(\bn)$. 

\medskip

\noindent
{\bf Proof of Theorem \ref{main2} (a) and (b).}

\begin{proof}

Let us denote $\mathcal{H}_c^1(\bn)=\{ u\in \mathcal{H}^1(\bn)| \ supp \ u \ \mbox{is compact} \}$.

\medskip

\noindent
{\bf Step 1:} We claim that for $u>0$ and $u \in\mathcal{H}^1(\bn)$ there exists $\phi_n \in \mathcal{H}_c^1(\bn), \phi_n\geq 0$ and $\phi_n\to u $ in $\mathcal{H}^1(\bn)$.

\medskip

To prove the claim, we use the definition of $\mathcal{H}^1(\bn)$ to extract a sequence $\psi_n \in C_c^{\infty}(\bn)$ such that $\psi_n\to u $ in $\mathcal{H}^1(\bn)$.
Then, up to a subsequence we have $\psi_n\to u $ a.e.
Now, let $\phi_n=\psi_n^+$. Since $\psi_n\in C_c^{\infty}(\bn)$, it is easy to see that $\psi_n^+, \psi_n^-\in \mathcal{H}_c^1(\bn)$. Further, $\|\psi_n\|_{\la_1}=\|\psi_n^+\|_{\la_1}+\|\psi_n^-\|_{\la_1}$. Therefore $\|\phi_n\|_{\la_1}\leq \|\psi_n\|_{\la_1} $.
This implies that $\limsup \|\phi_n\|_{\la_1}\leq \|u\|_{\la_1}$. Hence up to a subsequence $\phi_n\rightharpoonup v$  in $\mathcal{H}^1(\bn)$. By Rellich-Kondrachov compactness theorem and  $u>0$, we have that up to a subsequence $\phi_n\to v$ a.e. as well as $\phi_n\to u$ a.e. and hence $v=u$. Now by weak lower semi continuity of norm we have,
\begin{align*}
\|u\|_{\la_1}\leq \liminf \|\phi_n\|_{\la_1}
\leq \limsup \|\phi_n\|_{\la_1}\leq\|u\|_{\la_1}.
\end{align*}  
Hence $\phi_n\to u$ in $\mathcal{H}^1(\bn)$. This  completes the proof of the claim.

\medskip

By density, the weak formulation holds for all test functions $\phi \in \mathcal{H}^1_c(\bn).$ 
\medskip

\noindent
{\bf Step 2:} Let $\phi_n\in \mathcal{H}^1(\bn)$ be a sequence as in Step 1, satisfying $\phi_n\to u$ in $\mathcal{H}^1(\bn)$and $\phi_n\geq 0$ where $u$ is a positive $\mathcal{H}^1$ solution of equation \eqref{LSE}. Therefore plugging $\phi_n$ into the weak formulation we get,
\begin{align*}
\langle u,\phi_n\rangle_{\la_1}+(\la_1-\la)\inth u\phi_n\dx =\inth u^p\phi_n\dx +\theta\inth \phi_n u\ln u^2\dx.
\end{align*} 
Now choosing $\delta<1$ such that $\la_1-\la-\theta \ln \delta <0$, we can rewrite the equation as follows
\begin{align*}
&\langle u,\phi_n\rangle_{\la_1}+(\la_1-\la-\theta \ln \delta )\inth u\phi_n \dx +(-\theta)\int_{\{u>\delta\}} \phi_n u\ln \left(\frac{u^2}{\delta^2}\right)\dx\\
=&\inth u^p\phi_n \dx 
+\theta\int_{\{u\leq\delta\}} \phi_n u\ln \left(\frac{u^2}{\delta^2}\right)\dx.
\end{align*} 
Neglecting the term $(\la_1-\la-\theta \ln \delta )\inth u\phi_n$ we get,
\begin{align*}
\langle u,\phi_n\rangle_{\la_1} +(-\theta)\int_{\{u>\delta\}} \phi_n u\ln \left(\frac{u^2}{\delta^2}\right)\dx\geq\inth u^p\phi_n\dx +\theta\int_{\{u\leq\delta\}} \phi_n u\ln \left(\frac{u^2}{\delta^2}\right)\dx.
\end{align*} 
Passing the limit $n\to \infty$ in L.H.S. and using Fatou's lemma in R.H.S. we get,
\begin{align*}
 \inth u^2\ln \left(\frac{u^2}{\delta^2}\right)\dx\lesssim_{\delta} \|u\|^2_{\la_1}.
\end{align*}
This implies $u\in L^2(\bn)$ and hence by Lemma \ref{A lower bound}, the desired lower bound follows,
 proving both (a) and (b) simultaneously. This completes the proof. 
\end{proof}

\begin{Rem}
It follows from the proof of Theorem \ref{main2}(b) and radial decay of $H^1(\bn)$ functions that if $u \in \mathcal{H}^1(\bn)$ is a positive radial solution to \eqref{LSE} with $\theta <0$ and $\la \in \R,$ then the following precise decay estimate holds:
\begin{align*}
C_0\left(\sinh\frac{d(0,x)}{2}\right)^{-(N-1)}\leq u(x)\leq C_1\left(\cosh \frac{d(0,x)}{2}\right)^{-(N-1)}\quad , \forall x\in \bn\setminus B_{R_0}.
\end{align*}
for some constants $C_0,C_1 >0,$ with $C_1$ depending on $u.$
It may seem that recording such a growth estimate for solutions that do not exist is meaningless, but this is indicative of solutions belonging in $H^1_{loc}(\bn)\setminus \mathcal{H}^1(\bn)$ that admit such matching lower and upper bounds for some exponent $\alpha>0$. In the next subsection we show that even solutions satisfying such asymptotic decay do not exist.   
\end{Rem}

\subsection{Further remarks on the non-existence results.} In this subsection, we demonstrate that for $\theta <0,$
there does not exist a positive solution satisfying a strong type asymptotic decay $u(x) \approx (1 - |x|^2)^{\alpha},$ $\alpha>0.$ The hypothesis is certainly very strong, however, it is interesting that $\alpha$ could be arbitrarily small positive number. In that respect we thought to include this observation as a lemma.

\begin{Lem}
Let $\la \in \mathbb{R},$ and either $N \geq 3, 1 < p \leq 2^{\star} - 1,$ or $N=2, 1 < p < \infty$ and assume that $\theta < 0.$  Then there exists no positive solution satisfying the asymptotic  $u(x) \approx (1 - |x|^2)^{\alpha},$ for some $\alpha>0.$
\end{Lem}

\begin{proof} 
Assume the positive solution $u$ has the decay 
\begin{align*}
u(x) \approx (1 - |x|^2)^{\alpha}, \ \ \mbox{for some} \ \alpha >0,
\end{align*}
and the constant in $\approx$ may also depend on $u.$

\medskip

\noindent
{\bf Step 1:} We start with a positive subsolution of the following equation with sufficiently fast decay. Denote
   $V(x) = \left(\cosh \frac{d(0,x)}{2}\right)^{-c}, \ c>0.$ Then $V$ satisfies   

  \begin{align}\label{Sec5eqn1}
  -\De_{\bn}V-\gamma V  \lesssim V^q \quad on \ \,\bn, 
  \end{align}
  where $q=\frac{c+2}{c}$ and $\gamma \in \mathbb{R}$ %and all the parameters 
  depends only on the dimension, and $c$ large to be determined later.  We also define $V[z](x) = V \circ \tau_{-z}(x).$
 Set a cutoff  $\phi\equiv 1$ in $B_r$, $\phi \equiv 0$ in $\bn\setminus B_{r+1}$, $0\leq \phi\leq 1$ and $|\nabla_{\bn} \phi|_g\approx 1$.
 We use the test function $\phi V[z]$ in the weak formulation, to obtain,
 \begin{align*}
 &\inth \langle\nabla_{\bn} u,\nabla_{\bn} (\phi V[z])\rangle_g-\gamma u(\phi V[z])\dx\\
 =&\inth u^p\phi V[z] + (\la-\gamma)u\phi V[z]+(\theta\ln u^2)u\phi V[z]\dx .
\end{align*}
By integration by parts,
\begin{align}\label{Sec5eqn2}
&\inth u \left[-\De_{\bn}(\phi V[z])-\gamma \phi V[z] \right]\dx\nonumber\\ 
=&\inth u^p\phi V[z]+(\la-\gamma)u \phi V[z]+\theta  u\ln u^2\phi V[z]\dx.
\end{align}
 Expanding the L.H.S., we get,
 \begin{align*}
 &\inth u\left[-\De_{\bn} V[z] -\gamma  V[z] \right]\phi\dx - \int_{B_{r+1}\setminus B_{r}} u\langle\nabla_{\bn} \phi,\nabla_{\bn} V[z]\rangle_g\dx \\
 -&\int_{B_{r+1}\setminus B_{r}} u V[z] \De_{\bn} \phi\dx.
\end{align*}
Note that in the limit the last two terms vanish, thanks to the enough decay of $V[z].$
 Letting $r\to \infty,$ \eqref{Sec5eqn2} becomes
\begin{align}
&\inth u\left[-\De_{\bn} V[z] - \gamma V[z] \right]\dx\nonumber\\
=&\inth u^p V[z]\dx+(\la-\gamma)\inth u  V[z]\dx+\theta \inth u\ln u^2 V[z]\dx.
\end{align}
 Using \eqref{Sec5eqn1} and $u$ a positive solution, we estimate
\begin{align}
\inth u V[z]^q\dx \gtrsim \inth u^p V[z]\dx+(\la-\gamma)\inth u  V[z]\dx+\theta \inth u\ln u^2 V[z]\dx.
\end{align}
Dropping the positive p-th order non linear term $\inth u^pV[z]$ in R.H.S., we have
\begin{align}\label{inequality1}
\inth u V[z]^q\dx \gtrsim (\la-\gamma)\inth u  V[z]\dx+\theta \inth u\ln u^2 V[z]\dx.
\end{align}

\noindent
{\bf Step 2:} Interaction estimates.

\medskip

Now assume $|z|\in[\frac{1}{2},1)$ and denote $z^\star=\frac{z}{|z|^2}.$ Then $|x-z^{\star}|\geq 1 - |x| \approx (1 - |x|^2),$ for all $x \in B_1.$ We now estimate one by one. We start with 
\begin{align*}
\inth u V[z]^q\dx &=\inth u \circ \tau_{z} \ V^q\dx\\
&\lesssim(1-|z|^2)^{\alpha}\int_{B_1}\frac{(1-|x|^2)^{\alpha+cq-N}}{|x-z^{\star}|^{2\alpha}}dx\\
&\lesssim (1-|z|^2)^{\alpha}\int_{B_1}{|x-z^{\star}|^{cq- \alpha - N}}dx\\
&\lesssim(1-|z|^2)^{\alpha}\int_{B_3(z^{\star})}{|x-z^{\star}|^{cq- \alpha - N}}dx\\
&\lesssim (1-|z|^2)^{\alpha},
\end{align*}
where the constant in $\lesssim$ is independent of $z$. Here $c$ is chosen so that we used  $\alpha+cq-N>0, cq- \alpha - N < N.$ Next we handle the log term.
For that we set 
\begin{align*}
&I_1 = (1 - |z|^2)^{\alpha}\ln (1-|z|^2) \int_{B_1} \frac{(1 - |x|^2)^{c + \alpha - N}}{|x - z^{\star}|^{2\alpha}} \ dx \ , \\
&I_2 = (1 - |z|^2)^{\alpha} \int_{B_1} \frac{(1 - |x|^2)^{c + \alpha - N}}{|x - z^{\star}|^{2\alpha}} \ln \left[\frac{(1 - |x|^2)}{|x - z^{\star}|^2}\right] \ dx \ ,
\end{align*}
and $I = I_1 + I_2.$ Since $ u \circ \tau_{z}(x) \leq C_2 e^{-\alpha d(x,z)}$
for some $C_2>0$ and $\theta < 0$ we see that 
\begin{align*}
2\theta I + 2\theta C_2 \inth u \circ \tau_{z} \ V\dx \leq \theta \inth u \ln u^2 \ V[z]\dx.
\end{align*}

Next we derive a lower bound on $I.$ 

\medskip

\noindent
{\bf First we estimate $ I_2.$} Again thanks to the enough decay of $V,$ the estimates are relatively straight forward. In the following the only subtlety is where $(1 - |x|)$ and $|x- z^{\star}|$ are small.
\begin{align*}
| I_2 |&\leq  
  (1 - |z|^2)^{\alpha} \int_{B_1} \frac{(1 - |x|^2)^{c + \alpha - N}}{|x - z^{\star}|^{2\alpha}}\left(| \ln (1 - |x|^2) |+ |\ln |x - z^{\star}|^2| \right)\ dx \\
 &\lesssim  
  (1 - |z|^2)^{\alpha} \int_{B_1} \frac{(1 - |x|^2)^{c + \alpha - N}}{|x - z^{\star}|^{2\alpha}}\left((1 - |x|^2)^{-\delta} +  |x - z^{\star}|^{-\delta} + O(1)\right)\ dx \\
 &\lesssim (1 - |z|^2)^{\alpha}, 
\end{align*}
where $\delta>0$ is small, and $O(1)$ is independent of $z.$
 Here we used  $\ln t \lesssim t^{ \delta}$ for $t$ large and $\delta >0$ small, and $c + \alpha - N > 0.$ 
 Here $\delta$ is chosen so that $c - \frac{N-1}{2} - N - \delta>-N,$ so that the integral is uniformly bounded as $|z| \rightarrow 1$.
 Combining all we get the lower bound:
 \begin{align}\label{I_2}
\theta I_2 \gtrsim \theta (1 - |z|^2)^{\alpha}.
\end{align}

\medskip

\noindent
{\bf Now we estimate $\theta I_1.$} As before 
\begin{align*}
\inf_{|z| \in (\frac{1}{2},1)}\int_{B_1} \frac{(1 - |x|^2)^{c + \alpha - N}}{|x - z^{\star}|^{2\alpha}} \ dx >0.
\end{align*}
The upper bound follows from the same argument, while the lower bound is just an application of Fatou's lemma. Since $\theta \ln (1 - |z|^2) \geq 0$ we conclude

\begin{align}\label{I_1}
\theta I_1 \gtrsim  \theta (1 - |z|^2)^{\alpha}\ln (1 - |z|^2).
\end{align}

 Combining \eqref{I_1} and \eqref{I_2}, we get 
\begin{align*}
\theta \inth u \ln u^2 \ V[z]\dx - 2C_2\theta \inth u V[z]\dx \gtrsim \theta (1 - |z|^2)^{\alpha}\ln (1 - |z|^2),
\end{align*}
as $|z| \rightarrow 1.$ 

Finally, the estimate of $ \inth u V[z]\dx $ is same as before and is of order $(1 - |z|^2)^{\frac{N-1}{2}}.$

\medskip

\noindent
{\bf Step 3.} Final step.
Now combining the estimates obtained step 2, and putting in the inequality \ref{inequality1} and dividing by $(1-|z|^2)^{\alpha}$ we get
\begin{align*}
C \geq -|\la-\gamma+C\theta| + C\theta\ln(1-|z|^2),
\end{align*}
where $C$ is a positive constant independent of $z.$ This gives a contradiction as $|z|\to 1$ completing the proof of non-existence of solutions.
\end{proof}

\section{Appendix}

We include a few details that were left out during the proof of non-existence results.

\medskip

\noindent
{\bf Proof of Lemma \ref{sub}.}

\begin{proof}
 Recall that $\rho(x)=d(0,x)=\ln\bigg(\frac{1+|x|}{1-|x|}\bigg).$ For simplicity we shall denote $u(x) = u(\rho).$ A straightforward computation gives
$u(\rho) \ = \ (\sinh\frac{\rho}{2})^{-(N-1)},$
$u'(\rho) =-\frac{N-1}{2} (\sinh\frac{\rho}{2})^{-N}\cosh\frac{\rho}{2},$ $
u''(\rho)=\frac{N(N-1)}{4}\sinh\frac{\rho}{2}^{-(N+1)}\cosh^2(\frac{\rho}{2})-\frac{N-1}{4}(\sinh\frac{\rho}{2})^{-N}\sinh\frac{\rho}{2}. 
$
As a result
\begin{align*}
(N-1)\coth \rho \ u'(\rho) &=-\frac{(N-1)^2}{2}\frac{\cosh\rho}{\sinh\rho} \left(\sinh\frac{\rho}{2}\right)^{-N}\cosh\frac{\rho}{2}\\
&=-\frac{(N-1)^2}{4}\frac{(2\cosh^2{\frac{\rho}{2}}-1)}{2\cosh\frac{\rho}{2}\sinh\frac{\rho}{2}}\left(\sinh\frac{\rho}{2}\right)^{-N}\cosh\frac{\rho}{2}\\
&
=-\frac{(N-1)^2}{2}\left(\cosh^2\frac{\rho}{2}\right)\left(\sinh\frac{\rho}{2}\right)^{-(N+1)}+\frac{(N-1)^2}{4}\left(\sinh\frac{\rho}{2}\right)^{-(N+1)},
\end{align*}
and hence
\begin{align*}
&-u''(\rho)- (N-1)\coth\rho \frac{du}{d\rho}(\rho)\\
=&-\frac{N(N-1)}{4}\left(\sinh\frac{\rho}{2}\right)^{-(N+1)}\left(\cosh^2 \frac{\rho}{2}\right)
+\frac{N-1}{4} \left(\sinh\frac{\rho}{2}\right)^{-(N-1)}\\
&+\frac{(N-1)^2}{2} \left(\cosh^2\frac{\rho}{2}\right)\left(\sinh \frac{\rho}{2}\right)^{-(N+1)}
 -\frac{(N-1)^2}{4}\left(\sinh\frac{\rho}{2}\right)^{-(N+1)}\\
=&\left[\frac{(N-1)^2}{2}-\frac{N(N-1)}{4}\right]\left(\sinh\frac{\rho}{2}\right)^{-(N+1)}\left(\cosh^2\frac{\rho}{2}\right) +\frac{N-1}{4}\left(\sinh\frac{\rho}{2}\right)^{-(N-1)}\\
&-\frac{(N-1)^2}{4}
\left(\sinh\frac{\rho}{2}\right)^{-(N+1)}\\
=&\frac{(N-1)(N-2)}{4}\left(\sinh\frac{\rho}{2}\right)^{-(N+1)} \left(\cosh^2\frac{\rho}{2}\right)
+\frac{(N-1)}{4}\left(\sinh\frac{\rho}{2}\right)^{-(N-1)}\\
&-o\left(\left(\sinh\frac{\rho}{2}\right)^{-(N-1)}\right)\\
=&\frac{(N-1)}{4}\left((N-2)\coth^2{\frac{\rho}{2}}+1)(\sinh\frac{\rho}{2}\right)^{-(N-1)}-o\left(\left(\sinh\frac{\rho}{2}\right)^{-(N-1)}\right).
\end{align*}
Now, since $\frac{4}{N-1}\la>{N-1}$, we have,
\begin{align*}
-\De_{\bn} u-\la u&=\frac{(N-1)}{4}\left((N-2)\coth^2\frac{\rho}{2}+1-\tilde{\la}\right)\left(\sinh\frac{\rho}{2}\right)^{-(N-1)}-o\left((\sinh\frac{\rho}{2})^{-(N-1)}\right)\\
&<0,
\end{align*}
for all $\rho>\rho_{\la}$ where $\tilde{\la}=\frac{4}{(N-1)}\la$ and $(N-2)\coth^2\frac{\rho_{\la}}{2}+1<\tilde{\la}$.
\end{proof}

\vspace{1 cm}

\noindent
{\bf Proof of Lemma \ref{epsilon-pertubed-subsoln}.}

\begin{proof}
The notation $\rho$ is the same as in the Lemma \ref{sub} and recall $u_{\epsilon}(\rho)=(\sinh\frac{\rho}{2})^{-(N-1+\epsilon)}$, $f(\rho,\epsilon)=-u_{\epsilon}^{\prime\prime}(\rho)- (N-1)\coth\rho \ u_{\epsilon}^{\prime}(\rho)-\la u_\epsilon(\rho)$. Now computing the first and second derivatives, we get
\begin{align*}
u_{\epsilon}'(\rho)&=-\frac{(N-1+\epsilon)}{2}\left(\sinh\frac{\rho}{2}\right)^{-(N+\epsilon)}\cosh{\frac{\rho}{2}}\\
u_{\epsilon}''(\rho)&=\frac{(N-1+\epsilon)}{2}\frac{(N+\epsilon)}{2}\left(\sinh\frac{\rho}{2}\right)^{-(N+1+\epsilon)}\cosh^2{\frac{\rho}{2}}-\frac{N-1+\epsilon}{4} \left(\sinh{\frac{\rho}{2}}\right)^{-(N-1+\epsilon)}.\\%\sinh\frac{\rho}{2}\\
\end{align*}
As a result
\begin{align*}
(N-1)(\coth\rho)u'_{\epsilon}(\rho)
&=-\frac{(N-1)(N-1+\epsilon)}{2} {\frac{\cosh\rho}{\sinh\rho}} \left(\sinh\frac{\rho}{2}\right)^{-(N+\epsilon)} \cosh\frac{\rho}{2}\\
&=\frac{-(N-1)(N-1+\epsilon)}{2}\frac{2\cosh^2\frac{\rho}{2}-1}{2\cosh\frac{\rho}{2}\sinh\frac{\rho}{2}} \left(\sinh\frac{\rho}{2}\right)^{-(N+\epsilon)}\cosh\frac{\rho}{2},
\end{align*}
and hence,
\begin{align*}
&-u''_{\epsilon}(\rho)-(N-1)(\coth{\rho})u'_{\epsilon}(\rho)-\la u_{\epsilon}(\rho) \\
=&-\frac{(N-1+\epsilon)}{2}\frac{(N+\epsilon)}{2}\left(\sinh\frac{\rho}{2}\right)^{-(N+1+\epsilon)}\left(\cosh^2\frac{\rho}{2}\right)
+\frac{N-1+\epsilon}{4}\left(\sinh\frac{\rho}{2}\right)^{-(N-1+\epsilon)}\\%(\sinh(\frac{\rho}{2})\\
& +\frac{(N-1)(N-1+\epsilon)}{2}\left(\sinh\frac{\rho}{2}\right)^{-(N+\epsilon+1)}\left(\cosh^2\frac{\rho}{2}\right) \\
& -\frac{(N-1)(N-1+\epsilon)}{2}\left(\sinh\frac{\rho}{2}\right)^{-(N+1+\epsilon)}-\la\left(\sinh\frac{\rho}{2}\right)^{-(N-1+\epsilon)}\\
=&-\frac{(N-1)}{2}\frac{N}{2}\left(\sinh\frac{\rho}{2}\right)^{-(N+1+\epsilon)}\left(\cosh^2\frac{\rho}{2}\right)
-\frac{\epsilon}{2}\frac{N-1}{2}\left(\sinh\frac{\rho}{2}\right)^{-(N+1+\epsilon)}\left(\cosh^2(\frac{\rho}{2})\right)\\
&  -\frac{\epsilon^2}{4} \left(\sinh\frac{\rho}{2}\right)^{-(N+1+\epsilon)}\left(\cosh^2(\frac{\rho}{2})\right)
-\frac{\epsilon}{2}\frac{N}{2}\left(\sinh\frac{\rho}{2}\right)^{-(N+1+\epsilon)}\cosh^2\left(\frac{\rho}{2}\right)\\
&+\frac{N-1}{4} \left(\sinh \frac{\rho}{2}\right)^{-(N-1+\epsilon)}
+\frac{\epsilon}{4}\left(\sinh(\frac{\rho}{2}\right)^{-(N-1+\epsilon)}\\
&+\frac{(N-1)^2}{2} \left(\sinh\frac{\rho}{2}\right)^{-(N+\epsilon+1)}\left(\cosh^2\frac{\rho}{2}\right)
+\epsilon\frac{(N-1)}{2}(\sinh\frac{\rho}{2})^{-(N+1+\epsilon)} \cosh^2{\frac{\rho}{2}}\\
&-\frac{(N-1)^2}{2}(\sinh\frac{\rho}{2})^{-(N+1+\epsilon)}-\frac{\epsilon}{2}(N-1)\left(\sinh\frac{\rho}{2})\right)^{-N+1+\epsilon} 
-\la \left(\sinh\frac{\rho}{2}\right)^{-(N-1+\epsilon)}.
\end{align*}
\begin{comment}
Now,
\begin{align*}
(\sinh\frac{\rho}{2})^{-(\epsilon)}&=(\frac{(\e^\frac{\rho}{2}-\e^{-\frac{\rho}{2}}}{2})^{-\epsilon}=\frac{e^{-\frac{\epsilon\rho}{2}}}{2^{-\epsilon}}(1+\epsilon\e^{-\rho}+o(e^{-\rho}))\\
\cosh^2(\frac{\rho}{2})^{-\epsilon}&=\frac{\e^{\frac{\rho}{2}}+\e^{-\frac{\rho}{2}}}{2}^{-\epsilon}= (\frac{(\e^\frac{\rho}{2}-\e^{-\frac{\rho}{2}}}{2})^{-\epsilon}=\frac{e^{-\frac{\epsilon\rho}{2}}}{2^{-\epsilon}}(1+\epsilon\e^{-\rho}+o(e^{-\rho})).
\end{align*}
\end{comment}
Now, clubbing the terms together we get,
\begin{align*}
f(\rho,\epsilon)=(\sinh\frac{\rho}{2})^{-\epsilon}f(\rho,0)+\epsilon O\left(\sinh(\frac{\rho}{2})^{-(N-1+\epsilon)}\right)<0.
\end{align*}
whenever $\rho>R_{\la}$ and $\epsilon<<1$.
\end{proof}

\medskip

\medskip

\noindent
{\bf Acknowledgement.}  D. Karmakar acknowledges the support of the Department of Atomic Energy, Government of India, under project no. 12-R\&D-TFR-5.01-0520. 

\medskip

\noindent
{\bf Competing interests.} The authors have no competing interests to declare that are relevant to the content of this article.

\medskip

\noindent
{\bf Data availability statement.} Data sharing not applicable to this article as no datasets were generated or analysed during the current study.

\label{Bibliography}

\bibliography{BNlog}{} % The references (bibliography) information are stored in the file named "Bibliography.bib"

\bibliographystyle{alpha} % Use the "unsrtnat" BibTeX style for formatting the Bibliography

\end{document}